\documentclass[10pt]{article} 

\usepackage{url}
\usepackage[numbers]{natbib}
\usepackage[T1]{fontenc}
\usepackage{lmodern}    
\usepackage{microtype}
\usepackage{fix-cm}

\usepackage{mathtools}
\usepackage{amsmath}
\usepackage{amssymb}

\usepackage{amsthm}
\usepackage{aliascnt}
\usepackage{enumitem}
\usepackage{eurosym}

\usepackage{appendix}
\usepackage{color}
\usepackage[utf8]{inputenc}
\usepackage{geometry}
\usepackage{multirow}

\usepackage{cases}

\renewcommand{\emph}[1]{\textit{#1}}

\usepackage[unicode,breaklinks]{hyperref}
\usepackage[nameinlink,noabbrev]{cleveref}
\newcommand*{\email}[1]{\href{mailto:#1}{\nolinkurl{#1}}} 

\usepackage{todonotes}
\usepackage{xcolor}

\usepackage[english]{babel}
\usepackage[english=british]{csquotes}

\definecolor{linkred}{rgb}{0.7,0.15,0.15}
\definecolor{citegreen}{rgb}{0,0.5,0}
\definecolor{urlblue}{rgb}{0,0,0.7}
\hypersetup{colorlinks, linkcolor={urlblue},citecolor={citegreen}, urlcolor={linkred}}

\usepackage{xcolor,pifont}

\usepackage{subcaption}
\usepackage[justification=centering]{caption}
\makeatletter \@addtoreset{equation}{section}

\newtheorem{theorem}{Theorem}[section]

\newaliascnt{assumption}{theorem}

\aliascntresetthe{assumption}

\newaliascnt{proposition}{theorem}
\newtheorem{proposition}[proposition]{Proposition}
\aliascntresetthe{proposition}

\newaliascnt{definition}{theorem}
\newtheorem{definition}[definition]{Definition}
\aliascntresetthe{definition}

\newaliascnt{lemma}{theorem}
\newtheorem{lemma}[lemma]{Lemma}
\aliascntresetthe{lemma}

\newaliascnt{example}{theorem}

\aliascntresetthe{example}

\newaliascnt{corollary}{theorem}

\aliascntresetthe{corollary}

\newaliascnt{remark}{theorem}
\newtheorem{remark}[remark]{Remark}
\aliascntresetthe{remark}

\newaliascnt{condition}{theorem}

\aliascntresetthe{condition}

\crefname{theorem}{theorem}{theorems}
\Crefname{theorem}{Theorem}{Theorems}

\crefname{assumption}{assumption}{assumptions}
\Crefname{assumption}{Assumption}{Assumptions}

\crefname{proposition}{proposition}{propositions}
\Crefname{proposition}{Proposition}{Propositions}

\crefname{definition}{definition}{definitions}
\Crefname{definition}{Definition}{Definitions}

\crefname{lemma}{lemma}{lemmas}
\Crefname{lemma}{Lemma}{Lemmas}

\crefname{example}{example}{examples}
\Crefname{example}{Example}{Examples}

\crefname{corollary}{corollary}{corollaries}
\Crefname{corollary}{corollary}{Corollaries}

\crefname{remark}{remark}{remarks}
\Crefname{remark}{remark}{Remarks}

\crefname{condition}{condition}{conditions}
\Crefname{condition}{Condition}{Conditions}

\def \N{\mathbb{N}}
\def \P{\mathbb{P}}

\def \R{\mathbb{R}}

\def\Ac{{\cal A}}

\def\Oc{{\cal O}}

\newcommand{\smallertext}[1]{\text{\fontsize{5}{5}\selectfont$#1$}}
\newcommand{\smalltext}[1]{\text{\fontsize{4}{4}\selectfont$#1$}}

\definecolor{bleudefrance}{rgb}{0.19, 0.55, 0.91}

\definecolor{darkspringgreen}{RGB}{60, 179, 113}

\definecolor{viola}{RGB}{200, 0, 255}

\def\dbA{\mathbb{A}}

\def\dbE{\mathbb{E}}
\def\dbF{\mathbb{F}}

\def\dbL{\mathbb{L}}

\def\dbN{\mathbb{N}}
\def\dbP{\mathbb{P}}
\def\dbR{\mathbb{R}}

\def\a{\alpha}
\def\b{\beta}
\def\g{\gamma}
\def\d{\delta}
\def\e{\varepsilon}

\def\l{\lambda}
\def\m{\mu}
\def\n{\nu}
\def\si{\sigma}
\def\t{\tau}
\def\f{\varphi}
\def\th{\theta}
\def\o{\omega}

\def\G{\Gamma}

\def\Th{\Theta}
\def\L{\Lambda}
\def\Si{\Sigma}

\def\cA{{\cal A}}

\def\cC{{\cal C}}
\def\cD{{\cal D}}
\def\cE{{\cal E}}
\def\cF{{\cal F}}

\def\cH{{\cal H}}

\def\cL{{\cal L}}

\def\cN{{\cal N}}

\def\cQ{{\cal Q}}

\def\cS{{\cal S}}
\def\cT{{\cal T}}

\def\cV{{\cal V}}

\def\cZ{{\cal Z}}

\def\pa{\partial}

\def\ul{\underline}
\def\ol{\overline}

\def\bp{{\mathbf{p}}}

\def\bx{{\mathbf{x}}}

\def\bX{{\mathbf{X}}}

\def\1{\mathbf{1}}
\def\by{{\mathbf{y}}}
\def\bY{{\mathbf{Y}}}
\def\bZ{{\mathbf{Z}}}
\def\bz{{\mathbf{z}}}
\def\bv{{\mathbf{v}}}

\begin{document}

\title{\bf{Optimal control of Volterra integral diffusions and application to contract theory}}\author{Dylan Possamaï\footnote{Department of Mathematics, ETH Zurich, Switzerland, dylan.possamaï@math.ethz.ch. This author gratefully acknowledges support from the SNF project MINT 205121-21981.} \quad Mehdi Talbi\footnote{Laboratoire de Probabilités, Statistiques et Modélisation, Université Paris-Cité, France, talbi@lpsm.paris} }
\date{}

\maketitle

\begin{abstract}
This paper focuses on the optimal control of a class of stochastic Volterra integral equations. Here the coefficients are regular and not assumed to be of convolution type. We show that, under mild regularity assumptions, these equations can be lifted in a Sobolev space, whose Hilbertian structure allows us to attack the problem through a dynamic programming approach. We are then able to use the theory of viscosity solutions on Hilbert spaces to characterise the value function of the control problem as the unique solution of a parabolic equation on Sobolev space. We provide applications and examples to illustrate the usefulness of our theory, in particular for a certain class of time-inconsistent principal--agent problems. As a by-product of our analysis, we introduce a new Markovian approximation for Volterra-type dynamics.
\end{abstract}


\vspace{3mm}


\section{Introduction}


Let $(\Omega, \cF, \dbF, \dbP)$ denote a fixed probability space, endowed with a standard $(\dbF,\dbP)$--Brownian motion $W$ of dimension $d\in\dbN^\star$. We consider controlled stochastic Volterra integral equations of the form
\begin{align}\label{VolterraSDE}
X_t^\a = x + \int_0^t b_r(t, X_r^\a, \a_r)\mathrm{d}r + \int_0^t \si_r(t, X_r^\a, \a_r)\mathrm{d}W_r,
\end{align}
where $x \in \dbR^n$, $\a$ lives in an appropriate space of controls $\cA$ taking its values in some Polish space $A$, and
\[
b:[0,T]^2\times\dbR^n\times A\longrightarrow \dbR^n,
\;
\si:[0,T]^2\times\dbR^n\times A\longrightarrow M_{n,d}(\dbR)
\]
are continuous in all their variables. We are interested in the optimal control problem
\begin{align}\label{control-pb}
V_0(x) \coloneqq \sup_{\a \in \cA} \dbE^{\mathbb P}\bigg[\int_0^T f(t, X_t^\a, \a_t)\mathrm{d}t + g(X_T^\a)\bigg].
\end{align}

\medskip
Given their large scope of applications, stochastic control problems of the form \eqref{VolterraSDE}--\eqref{control-pb} have received strong attention in the scientific literature. They have, for example, raised interest in medical sciences (see \emph{e.g.} \citeauthor*{schmiegel2006self} \cite{schmiegel2006self} and \citeauthor*{saeedian2017memory} \cite{saeedian2017memory}) and in finance, in particular in the study of rough volatility models (see \emph{e.g.} \citeauthor*{bayer2016pricing} \cite{bayer2016pricing} and \citeauthor*{gatheral2018volatility} \cite{gatheral2018volatility}). More broadly, Volterra-type memory effects also arise in stochastic advertising and goodwill models with carryover or distributed forgetting, as well as in systems with hereditary effects such as stochastic heat equations with memory; see, for instance, \citeauthor*{giordano2024optimal} \cite{giordano2024optimal}, \citeauthor*{gozzi2024optimal} \cite{gozzi2024optimal}, and \citeauthor*{confortola2014optimal} \cite{confortola2014optimal}.

\medskip
Recently, \citeauthor*{hernandez2024time} showed in \cite{hernandez2024time} that Volterra-type control problems naturally arise in contracting problems involving some form of time-inconsistency. These problems have quite specific features and may be viewed as \emph{extended Volterra control problems}. More precisely, the state process is now an uncountable family of processes $\bX^\a \coloneqq (X^{\a,s})_{\{s \in [0,T]\}}$, with
\begin{equation}\label{ExtendedVolterraSDE}
 X_t^{\a,s}
 = x^s
 + \int_0^t b_r\big(s, X_r^{\a,r}, X_r^{\a,s}, \a_r\big)\mathrm{d}r
 + \int_0^t \si_r\big(s, X_r^{\a,r}, X_r^{\a,s}, \a_r\big)\mathrm{d}W_r.
\end{equation}
In particular, if there exist coefficients $\bar b$ and $\bar\si$ such that
\[
b_r(s,x,y,a)=\bar b_r(s,x,a),
\;
\si_r(s,x,y,a)=\bar\si_r(s,x,a),
\]
and if the initial profile is constant, \emph{i.e.} $x^s=x$ for all $s \in [0,T]$, then the diagonal process $t \longmapsto X_t^{\a,t}$ satisfies the classical stochastic Volterra equation \eqref{VolterraSDE}.

\medskip
Volterra-type control problems have also attracted strong attention for their challenging mathematical features. Indeed, due to the presence of the $t$ in $b$ and $\si$, it is well known that the optimisation problem \eqref{control-pb} is time inconsistent: $X^\a$ is not Markov or not even a semi-martingale in general, and therefore the flow property does not apply. Various techniques have been considered to overcome this difficulty. One rather popular method is to handle the problem through a maximum principle approach, see \emph{e.g.} \citeauthor*{agram2015malliavin} \cite{agram2015malliavin}, \citeauthor*{agram2018new} \cite{agram2018new}, \citeauthor*{lin2020controlled} \cite{lin2020controlled} and 
\citeauthor*{hamaguchi2023maximum} \cite{hamaguchi2023maximum}. We also mention the recent contribution of \citeauthor*{cardenas2022existence} \cite{cardenas2022existence}, who search for an optimal control in a relaxed form.   
 
\medskip
A large number of papers focus on recovering time-consistency by embedding the problem in a larger space, in which the new state process satisfies the usual flow property. These works are often referred to as using a \emph{lifting} approach. In the contributions of \citeauthor*{abi2021linear} \cite{abi2021linear}, \citeauthor*{di2023lifting} \cite{di2023lifting} and \citeauthor*{hamaguchi2023markovian} \cite{hamaguchi2023markovian}, the kernel of the stochastic Volterra integral equation is written as the linear transform of some element defined on an appropriate Banach space (and even in a Hilbert space in the case of \cite{hamaguchi2023markovian}). This linear transformation involves some semi-group structure, so that it is possible to write the state process $X^\a$ as the image of an infinite dimensional process $\bX^\a$ by the same transformation, with $\bX^\a$ satisfying an infinite dimensional stochastic differential equation (SDE for short) and therefore satisfying the Markov property. In a slightly different approach, \citeauthor*{viens2019martingale} \cite{viens2019martingale} lifts the state process---typically a fractional Brownian motion---in the Banach space of continuous path, treating the `Volterra time' (the $t$ in $b$ and $\si$ in \eqref{VolterraSDE}) as a parameter. This approach has been used in several subsequent works, such as the ones of \citeauthor*{wang2022path} \cite{wang2022path} and \citeauthor*{wang2023linear} \cite{wang2023linear}. 

\medskip
We propose a new lifting in the same spirit as \cite{viens2019martingale}, in the case of regular kernels. More precisely, we also treat the `Volterra time' as a parameter, thus lifting the state process in a space of paths. We shall however assume that the coefficients of \eqref{ExtendedVolterraSDE} are differentiable in this parameter in the Sobolev sense. This provides a Hilbertian structure as well as continuity in the parameter, which is crucial to connect the original problem \eqref{VolterraSDE}--\eqref{control-pb} with the lifted one. Our contribution has four main features. First, unlike semigroup-based lifts, it does not rely on any specific representation of the kernel. Second, the lift is performed in a Sobolev/Hilbert space, which is precisely what allows us to use the standard viscosity theory on Hilbert spaces; see \emph{e.g.} \citeauthor*{lions1988viscosity} \cite{lions1988viscosity, lions1989viscosity, lions1989viscosity2}. Third, the framework covers the extended Volterra systems arising in time-inconsistent contract theory, with the classical Volterra dynamics appearing as a special case. Finally, the same Hilbert structure yields a natural Markovian approximation of the Volterra state through projections on finite-dimensional subspaces.

\medskip
The paper is organised as follows. In \Cref{sec:infinite-dimension}, we introduce precisely the lifted space and process, and highlight their main properties. In \Cref{sec:DPE}, the value function of the infinite dimensional problem is characterised as the unique viscosity solution to a parabolic equation on Sobolev space. Special attention is also given to the case of uncontrolled volatility. We apply our theory to several examples in \Cref{sec:examples}: beyond the introductory linear--quadratic illustration, Section~4 now contains two self-contained genuinely Volterra applications, namely a regular propagator liquidation model and a stochastic advertising/goodwill model with memory, as well as the contracting problem with a sophisticated agent. \Cref{sec:markovian-representation} discusses an interesting by-product of our analysis, namely a new Markovian representation of stochastic Volterra processes. Finally, \Cref{sec:comparison} compares our contribution with some of the aforementioned references and \Cref{sec:singular-kernel} discusses the case of singular kernels.

\section{The infinite dimensional problem}\label{sec:infinite-dimension}

Our main requirements to define a `good' lifting are the following

\medskip
$\quad\bullet$ the state space must be a Hilbert space, and the lifted state process must satisfy some flow property, as this will enable us to apply the standard theory of viscosity solutions on Hilbert space for a large class of stochastic control problems;

\medskip
$\quad\bullet$ if the original process writes as \eqref{VolterraSDE}, it must write as a continuous function of the lifted process.  

\subsection{Choice of the state space}

Let $\dbL^2([0,T],\mathrm{d}t)$ denote the equivalence class of square-integrable functions $\f:[0,T]\longrightarrow\dbR^n$, and let $C_c^1((0,T))$ be the set of $C^1$ functions $\psi$ on $(0,T)$ such that $\psi$ and $\psi^\prime$ have compact support. Introduce the Sobolev space
\begin{align}\label{state-space}
 W^{1,2}([0,T]) \coloneqq \bigg\{ u \in \dbL^2([0,T], \mathrm{d}t) : \exists u^\prime \in \dbL^2([0,T], \mathrm{d}t),\; \int_0^T u(t)\f^\prime(t) \mathrm{d}t = \int_0^T u^\prime(t) \f(t) \mathrm{d}t,\; \forall \f \in C_c^1((0,T)) \bigg\},
\end{align}
as well as its scalar product
\[
\langle u, v \rangle_{W^{1,2}} \coloneqq \int_0^T u(t)v(t)\mathrm{d}t + \int_0^T u^\prime(t)v^\prime(t)\mathrm{d}t.
\]
Let $H$ be the space defined by
\[
H \coloneqq \big\{ \bx \coloneqq (\bx_1, \dots, \bx_n) : \bx_k \in W^{1,2}([0,T]),\; \forall k \in \{1, \dots, n\}\big\},
\]
endowed with the scalar product
\[
\langle \bx, \by \rangle_H \coloneqq \sum_{k=1}^n \langle \bx_k, \by_k \rangle_{W^{1,2}},
\]
and the corresponding norm $\lVert \bx \rVert_H \coloneqq \sqrt{\langle \bx, \bx \rangle_H}$. Then $(H,\lVert\cdot\rVert_H)$ is a Hilbert space. Moreover, we have the following compact embedding result.

\begin{lemma}\label{lem:injection}
The Sobolev space $H$ is continuously and compactly embedded into $\cC \coloneqq C^0([0,T],\dbR^n)$. In particular, there exists a constant $C\ge0$ such that
\[
 \lVert \bx \rVert_\infty \coloneqq \sup_{t \in [0,T]} |\bx^t| \le C \lVert \bx \rVert_H,\; \forall \bx \in H.
\]
\end{lemma}

\begin{proof}
We prove the result in the case $n=1$; the general case follows coordinate-wise. By Sobolev embedding, each element $\bx \in H$ has a unique continuous representative on $[0,T]$, which we still denote by $\bx$ (see, for instance, \citeauthor*{brezis2011functional} \cite[Theorem 8.2]{brezis2011functional}). For any $(s,t) \in [0,T]^2$
\[
|\bx^t| \le |\bx^s| + \int_0^T |\bx^\prime(r)|\mathrm{d}r.
\]
Integrating with respect to $s$ yields
\[
|\bx^t| \le \frac{1}{T}\int_0^T |\bx(r)|\mathrm{d}r + \int_0^T |\bx^\prime(r)|\mathrm{d}r,
\]
and therefore, by Cauchy--Schwarz
\begin{equation}\label{ineq:inf-sobolev}
\lVert \bx \rVert_\infty \le C \lVert \bx \rVert_H,
\end{equation}
for some constant $C$ depending only on $T$. This proves continuity of the embedding. Now let $(\bx_m)_{m\in\N}$ be bounded in $H$. By \eqref{ineq:inf-sobolev}, the sequence is uniformly bounded in $\cC$. Moreover, for all $(s,t) \in [0,T]^2$
\[
|\bx_m^t-\bx_m^s| \le \int_s^t |\bx_m^\prime(r)|\mathrm{d}r \le |t-s|^{1/2}\|\bx_m^\prime\|_{\dbL^2([0,T])},
\]
so $(\bx_m)_{m\in\N}$ is equicontinuous. By Arzelà--Ascoli, every bounded sequence in $H$ admits a subsequence converging in $\cC$. Hence the embedding $H\hookrightarrow \cC$ is compact.
\end{proof}

\begin{remark}
From now on, we identify each element of $H$ with its continuous representative. In particular, the evaluation map
\[
 \Phi : [0,T]\times H\ni(t,\bx)\longmapsto \bx^t\in\dbR^n,
\]
is continuous and even Lipschitz-continuous in $\bx$.
\end{remark}

\subsection{The infinite dimensional dynamics}

Let
\[
b:[0,T]^2\times\dbR^n\times\dbR^n\times A\longrightarrow \dbR^n,
\;
\si:[0,T]^2\times\dbR^n\times\dbR^n\times A\longrightarrow M_{n,d}(\dbR),
\]
and introduce the $H$-valued mapping
\[
 B(t, \bx, a)(s) \coloneqq b_t\big(s, \bx^t, \bx^s, a\big),
\]
as well as the $\cL(\dbR^d,H)$-valued mapping
\[
 \Si(t, \bx, a)(s) \coloneqq \si_t\big(s, \bx^t, \bx^s, a\big),
\]
for all $(t,\bx,a,s)\in[0,T]\times H\times A\times[0,T]$, where $\cL(\dbR^d,H)$ denotes the space of bounded linear maps from $\dbR^d$ to $H$.

\medskip
Let $(\Omega,\cF,\dbP)$ be a probability space endowed with a standard $d$-dimensional Brownian motion $W$, and denote by $\dbF^W$ its natural filtration. Denote by $\cA^W$ the set of $\dbF^W$--progressively measurable processes taking their values in $A$ such that the $H$-valued SDE
\begin{equation}\label{liftedSDE1}
\bX_t^\a = \bx_0 + \int_0^t B(r, \bX_r^\a, \a_r)\mathrm{d}r + \int_0^t \Si(r, \bX_r^\a, \a_r)\mathrm{d}W_r,\; t\in[0,T],\; \dbP\text{--a.s.},
\end{equation}
has a unique strong solution. 

\medskip
Our first result is intuitively clear but act as a sanity check by formalising the links between \Cref{liftedSDE1}, \Cref{ExtendedVolterraSDE}, and \Cref{VolterraSDE}.

\begin{proposition}\label{prop:lifting-equivalence}
For each $s\in[0,T]$, let $\Pi_s:H\longrightarrow\dbR^n$ be the evaluation map $\Pi_s(\bx)\coloneqq \bx^s.$ Then $\Pi_s$ is continuous. Moreover, let $\a\in\cA^W$, and let $\bX^\a$ be the $H$-valued solution of \eqref{liftedSDE1}. Then, for every $s\in[0,T]$, the process
\[
X_t^{\a,s}\coloneqq \Pi_s(\bX_t^\a),\; t\in[0,T],\; \a\in\Ac^W,
\]
satisfies
\begin{equation}\label{pointwise-family}
X_t^{\a,s}
 = x^s
 + \int_0^t b_r\big(s, X_r^{\a,r}, X_r^{\a,s}, \a_r\big)\mathrm{d}r
 + \int_0^t \si_r\big(s, X_r^{\a,r}, X_r^{\a,s}, \a_r\big)\mathrm{d}W_r.
\end{equation}

Conversely, assume that $(X^{\a,s})_{s\in[0,T]}$ is an $H$-valued family of processes such that \eqref{pointwise-family} holds for every $s\in[0,T]$. Then the $H$-valued process $\bX^\a$ defined by $\bX_t^\a(\cdot)\coloneqq X_t^{\a,\cdot}$ solves \eqref{liftedSDE1}.

\medskip
Finally, if there exist coefficients $\bar b$ and $\bar\si$ such that
\[
b_r(s,x,y,a)=\bar b_r(s,x,a),
\;
\si_r(s,x,y,a)=\bar\si_r(s,x,a),
\]
and if $\bX_0^\a=\bp(x)$ for some $x\in\dbR^n$, then the diagonal process $t\longmapsto X_t^{\a,t}$ solves \eqref{VolterraSDE}.
\end{proposition}

\begin{proof}
The continuity of $\Pi_s$ follows immediately from \Cref{lem:injection}. If $\bX^\a$ solves \eqref{liftedSDE1}, then applying $\Pi_s$ to both sides and using the definitions of $B$ and $\Si$ yields \eqref{pointwise-family}. Conversely, if the family $(X^{\a,s})_{s\in[0,T]}$ is $H$-valued and satisfies \eqref{pointwise-family} for every $s$, then the $H$-valued identity \eqref{liftedSDE1} follows by evaluating both sides at each $s\in[0,T]$. The last assertion is the particular case in which the coefficients do not depend on the second state variable and the initial profile is constant.
\end{proof}

Our next result provides concrete sufficient conditions guaranteeing the existence and uniqueness of a $H$-valued solution to \eqref{liftedSDE1}, at least whenever the control $\a$ is fixed. The point here is that standard Lipschitz continuity of the coefficients in the finite-dimensional variables is not sufficient in general, unlike in the purely finite-dimensional case.

\begin{proposition}\label{prop:SDE}
For $\phi \in\{ b, \sigma\}$, assume that there exists two functions $\phi^1$ and $\phi^2$ such that
\[
\phi_t(s, x, y, a) = \phi_t^1(s, x, a) + \phi_t^2(s, a)y,\; \mbox{\rm for all $(t, s, x, y, a) \in [0,T]^2 \times \dbR^n\times\R^n \times A$}. 
\]
Assume furthermore that

\medskip
$(i)$ $\phi^1$ is continuous in all its variables, Lipschitz-continuous in $x$ uniformly in $(t,s,a)$, and admits a Sobolev derivative w.r.t. $s$ which is also continuous in all its variables and Lipschitz-continuous in $x$ uniformly in $(t,s,a);$

\medskip
$(ii)$ $\phi^2$ is continuous and uniformly bounded in all its variables, and admits a Sobolev derivative w.r.t.\ $s$ which is also continuous and uniformly bounded in all its variables.

\medskip
Then \eqref{liftedSDE1} has a unique solution in $H$. 
\end{proposition}

\begin{proof}
We check the assumptions of \citeauthor*{gawarecki2011stochastic} \cite[Theorem 3.3]{gawarecki2011stochastic}. For simplicity, we argue in the case $n=1$; the multidimensional case is obtained coordinate-wise.

\medskip
Fix $(t,\bx,a)\in[0,T]\times H\times A$. By definition
\[
B(t,\bx,a)(s)=b_t^1(s,\bx^t,a)+b_t^2(s,a)\bx^s.
\]
Since $s\longmapsto b_t^1(s,\bx^t,a)$ and $s\longmapsto b_t^2(s,a)$ have Sobolev regularity by assumption, and since $\bx\in H$, it follows that $B(t,\bx,a)\in H$. The same argument applies to $\Si(t,\bx,a)$.

\medskip
We secondly verify that $B$ and $\Si$ have linear growth in $\bx$, uniformly in $(t,a)$. We denote by $\pa_s$ the derivation w.r.t. $s$ in the Sobolev sense. We have
\begin{align*}
\lVert B(t, \bx, a) \rVert_H^2 &= \int_0^T \big(b_t(s, \tilde \bx^t, \tilde \bx^s, a)\big)^2\mathrm{d}s + \int_0^T \big(\pa_s(b_t(s, \tilde \bx^t, \tilde \bx^s, a))\big)^2\mathrm{d}s \\
&= \int_0^T \big(b_t^1(s, \tilde \bx^t, a) + b_t^2(s,a) \tilde \bx^s\big)^2 \mathrm{d}s + \int_0^T \big(\pa_s b_t^1(s, \tilde \bx^t, a) + \pa_s b_t^2(s,a) \tilde \bx^s + b_t^2(s,a) (\widetilde{\pa_s \bx})^s \big)^2 \mathrm{d}s. 
\end{align*}
Since $b^1$ and $\pa_s b^1$ are Lipschitz-continuous in their space variable, uniformly in $(t,s,a)$, we have for some constant $C$
\[
 \lvert b_t^1(s, \tilde \bx^t, a) \rvert + \lvert \pa_s b_t^1(s, \tilde \bx^t, a) \rvert \le C(1 + \lvert \tilde \bx^t \rvert) \le C(1 + \lVert \bx \rVert_H), 
 \]
see \eqref{ineq:inf-sobolev} for the latter inequality. Then, we easily deduce from the boundedness of $b_t^2$ and $\pa_s b_t^2$ that $B$, and similarly $\Si$, have quadratic growth in $\bx$. 

\medskip
We finally prove that $B$ and $\Si$ are Lipschitz-continuous in $\bx$. Fix $(\bx, \by) \in H^2$, we have
\begin{align*}
\lVert B(t, \bx, a) - B(t, \by, a) \rVert_H^2 &\le \int_0^T \big(\lvert b_t^1(s, \tilde \bx^t, a) - b_t^1(s, \tilde \by^t, a) \rvert + | b_t^2(s, a) \tilde \bx^s - b_t^2(s, a) \tilde \by^s |  \big)^2 \mathrm{d}s \\
&\quad+ \int_0^T \big( | \pa_s b_t^1(s, \tilde \bx^t, a) - \pa_s b_t^1(s, \tilde \bx^t, a)| + | \pa_s b_t^2(s,a) \tilde \bx^s - \pa_s b_t^2(s,a) \tilde \bx^s |  \\
&\qquad + | b_t^2(s,a) \widetilde{\pa_s \bx}^s - b_t^2(s,a) \widetilde{\pa_s \by}^s |   \big)^2 \mathrm{d}s \\
&\le C\bigg( \int_0^T \big( | \tilde \bx^t - \tilde \by^t | + | \tilde \bx^s - \tilde \by^s | \big)^2 \mathrm{d}s + \int_0^T  \big(| \tilde \bx^s - \tilde \by^s | + | \widetilde{\pa_s \bx}^s - \widetilde{\pa_s \by}^s | \big)^2 \mathrm{d}s \bigg),
\end{align*}
for some constant $C>0$, where we used the Lipschitz-continuity of $b^1$ and $\pa_s b^1$ and the boundedness of $b^2$ and $\pa_s b^2$. Recalling \eqref{ineq:inf-sobolev}, we have
\[
 | \tilde \bx^t - \tilde \by^t | \le  \lVert \tilde \bx - \tilde \by \rVert_\infty \le \lVert  \bx^t -  \by^t \rVert_H, 
 \]
from which we finally deduce that $B$ is Lipschitz-continuous in $\bx$, uniformly in {\color{black}$(t,a)$.} We proceed similarly for $\Si$. 
\end{proof}

\begin{remark}\label{rem:weak-existence}
The linear dependence on the second state variable is used only to obtain a strong $H$-valued formulation with Lipschitz coefficients. Without this structure, one may still expect weak well-posedness, see, for instance, {\rm \cite[\S 3.9]{gawarecki2011stochastic}}. For singular kernels, there is by now a substantial recent literature on weak and strong well-posedness, including weak solutions for convolution kernels {\rm \citeauthor*{abijaber2021weak} \cite{abijaber2021weak}}, for more general kernels {\rm \citeauthor*{promel2023existence} \cite{promel2023existence}}, {\rm \citeauthor*{abijaber2025weak} \cite{abijaber2025weak}}, and strong or pathwise-uniqueness results in singular Hölder settings {\rm \citeauthor*{promel2023stochastic} \cite{promel2023stochastic}}, {\rm \citeauthor*{promel2025pathwise} \cite{promel2025pathwise}}, or {\rm \citeauthor*{hamaguchi2025weak} \cite{hamaguchi2025weak}}.

\medskip
Note also that when $\bX$ is the lifted version of {\rm \Cref{VolterraSDE}}, the term in $\bx^s$ is not involved in \eqref{liftedSDE1}, and therefore the well-posedness of \eqref{liftedSDE1} directly proceeds from the well-posedness of \eqref{VolterraSDE}. 

\end{remark}

\begin{remark}\label{rem:advantages}{
This infinite-dimensional process enjoys two important properties with respect to the original Volterra-type dynamics.

\medskip
$(i)$ The process $\bX^\a$ solves a stochastic differential equation, whereas the process $X^\a$ defined in \eqref{VolterraSDE} solves a stochastic integral equation. In particular, $\bX^\a$ is a semi-martingale.

\medskip
$(ii)$ For $x \in \dbR^n$, denote by $\bp(x)$ the element of $H$ such that $\bp(x)^t=x$ for all $t\in[0,T]$. Assume that $\bX_0^\a=\bp(x)$, and that the coefficients are of the classical form
\[
b_r(s,x,y,a)=\bar b_r(s,x,a),
\;
\si_r(s,x,y,a)=\bar\si_r(s,x,a).
\]
Then, by {\rm\Cref{prop:lifting-equivalence}}, the diagonal process $t\longmapsto X_t^{\a,t}$ satisfies \eqref{VolterraSDE}. Therefore, whenever uniqueness holds for the latter equation, the diagonal of the lifted process coincides with the original controlled dynamics. Existence and uniqueness for \eqref{VolterraSDE} under standard Lipschitz-continuity assumptions go back, for instance, to {\rm\citeauthor*{ito1979existence} \cite{ito1979existence}}, {\r,\citeauthor*{protter1985volterra} \cite{protter1985volterra}}, and {\rm\citeauthor*{pardoux1990stochastic} \cite{pardoux1990stochastic}}.
}
\end{remark}

\subsection{The infinite dimensional control problem}

From now on, for $\bx \in H$, we shall abuse the notation and still denote by $\bx$ its continuous representative.  Given the dynamics \eqref{liftedSDE1}, we consider the control problem
\begin{align}\label{liftedcontrolpb}
\cV_0(\bx_0) \coloneqq  \sup_{\a \in \cA^\smalltext{W}} \dbE\bigg[\int_0^T F(r, \bX_r^\a, \a_r)\mathrm{d}r + G(\bX_T^\a)\bigg],
\end{align}
for some $F : [0,T] \times H \times A \longrightarrow \dbR$ and $G : H \longrightarrow \dbR$. The following result states that this  control problem is connected to the control problem  of a Volterra-type SDE \eqref{control-pb} in the following way.
\begin{proposition}\label{prop:aux-to-or}
Let $x \in \dbR^n$. Assume $b$ and $\si$ write as in {\rm\Cref{VolterraSDE}}, and that $F(t, \bx,a ) =  f(t, x^t, a)$ and $G(\bx) = g(\bx^T)$ for all $(t, \bx, a) \in [0,T] \times H \times A$, with $f$ and $g$ as in the original problem \eqref{control-pb}. Assume furthermore that uniqueness holds for \eqref{VolterraSDE}. Then $\cV_0(\bp(x)) = V_0(x)$. 
\end{proposition}

\begin{proof}
By \Cref{prop:lifting-equivalence}, when the initial condition of the lifted process is $\bp(x)$ and the coefficients are of the classical form, the diagonal process of $\bX^\a$ coincides with the original Volterra dynamics $X^\a$, $\dbP$--a.s. Since $F(t,\bx,a)=f(t,\bx^t,a)$ and $G(\bx)=g(\bx^T)$, and since both problems optimise over the same set of open-loop controls, the two value functions are equal. 
\end{proof}

\begin{remark}[Choice of the set of controls]{
The choice of the set of controls is crucial for the above proposition. Indeed, if we choose to consider closed-loop controls for either \eqref{control-pb} or \eqref{liftedcontrolpb}, then we might have $\cV_0(\bp(x)) \neq V_0(x)$, as the filtration generated by $X$, the filtration generated by $\bX$ and $\dbF^W$ are different in general. However, if we consider the infinite dimensional control problem \eqref{liftedcontrolpb} as an object on its own---for example motivated by the study of moral hazard questions for time-inconsistent agents, see {\rm\cite{hernandez2024time}}---then we may either consider closed-loop or open-loop controls: both cases can be encapsulated in our dynamic programming approach.}  
\end{remark}

\section{Dynamic programming equation}\label{sec:DPE}

\subsection{The value function}

We introduce a dynamic version of the control problem \eqref{liftedcontrolpb}. Denoting by $\bX^{t, \bx, \a}$ the solution of \eqref{liftedSDE1} such that $\bX_t^{t, \bx, \a} = \bx$, we define:
\begin{align}\label{value-function}
\cV(t,\bx) \coloneqq  \sup_{\a \in \cA} \dbE\bigg[\int_t^T F(r, \bX_r^{t,\bx,\a}, \a_r)\mathrm{d}r + G(\bX_T^{t,\bx,\a})\bigg], \; \mbox{for all $(t,\bx) \in [0,T] \times H$}.
\end{align}

\begin{proposition}[Regularity of the value function]\label{prop:regularity}
Assume that

\medskip
$(i)$ $F$ is uniformly continuous in $(t, \bx) \in [0,T] \times B_H(0,R)$, uniformly in $a \in A$, for all $R \ge 0$, where $B_H(0,R)$ denotes the ball of radius $R$ and centre $0$ for the metric $\lVert \cdot \rVert_H;$

\medskip
$(ii)$ $G$ is uniformly continuous in $\bx \in B_H(0,R)$ for all $R \ge 0;$

\medskip
$(iii)$ $F$ and $G$ have polynomial growth in $\bx$, uniformly in the other variables.

\medskip
Then the value function $\cV$ is uniformly continuous on all the sets $[0,T] \times B_H(0,R)$, $R \ge 0$, and has polynomial growth in $\bx$ uniformly in $t$.
\end{proposition}

\begin{proof}
This is a direct application of \cite[Proposition 3.61]{fabbri2017stochastic}.
\end{proof}

Our lifted control problem \eqref{liftedcontrolpb} falls under the scope of Markovian control problems on Hilbert spaces, and we may therefore naturally formulate the following dynamic programming principle.
\begin{proposition}[Dynamic programming principle]
Under the assumptions of {\rm\Cref{prop:regularity}}, we have
\begin{align}\label{DPP}
\cV(t, \bx) = \sup_{\a \in \cA} \dbE\bigg[\int_t^\th F\big(r, \bX_r^{t, \bx, \a}, \a_r\big)\mathrm{d}r + \cV(\th, \bX_{\th}^{t, \bx, \a})\bigg],\; \forall (t,\bx) \in [0,T] \times H,\; \text{\rm and}\; \th \in \cT_{t,T},
\end{align}
where $\cT_{t,T}$ denotes the set of $[t,T]$-valued $\dbF$--stopping times.
\end{proposition}
\begin{proof} 
This is a direct application of \cite[Proposition 2.24]{fabbri2017stochastic}.
\end{proof}

\subsection{Viscosity solutions}

For any smooth $\f : [0,T] \times H \longmapsto  \dbR$, we denote by $\pa_t \f$ the derivative of $\f$ with respect to $t \in [0,T]$, and by $D_\bx \f$ and $D_{\bx \bx}^2 \f$ the first- and second-order Fréchet derivatives of $\f$ with respect to $\bx \in H$. For all $(t, \bx) \in [0,T] \times H$, by Riesz's representation theorem, $D_\bx \f(t, \bx)$ can be identified to an element of $H$, and $D_{\bx \bx}^2 \f(t,\bx)$ to an endomorphism of $H$.  

\medskip
The purpose of this section is to show that the value function $\cV$ of \eqref{liftedcontrolpb} can be characterised as the unique viscosity solution of the dynamic programming equation
\begin{align}\label{DPE}
-\pa_t u(t,\bx) -  \sup_{a \in A}\bigg\{ \big\langle D_\bx u(t, \bx), b_t(\cdot, \tilde \bx^t, \bx^\cdot, a) \big\rangle_H + \frac 1 2 \big\langle \si_t(\cdot, \tilde \bx^t, \bx^\cdot, a), D_{\bx \bx}^2 u(t,\bx)\si(\cdot, \tilde \bx^t, \bx^\cdot, a)\big\rangle_H + F(t, \bx, a)  \bigg\} = 0, 
\end{align}
with terminal condition $u|_{t=T} = G$.

\begin{definition}[Viscosity solutions]
Let $u : [0,T] \times H \longrightarrow \dbR$ be locally bounded.

\medskip
{$(i)$} $u$ is said to be a viscosity super-solution of \eqref{DPE} if $u(T, \cdot) \ge G$ and, for all $\f \in C^{1,2}([0,T] \times H)$ such that $u-\f$ has a local minimum in $(t,\bx)$, we have
\[
-\pa_t \f(t,\bx) -  \sup_{a \in A}\bigg\{ \big\langle D_\bx \f(t, \bx), b_t(\cdot, \tilde \bx^t, \bx^\cdot, a) \big\rangle_H \nonumber + \frac 1 2 \big\langle \si_t(\cdot, \tilde \bx^t, \bx^\cdot, a), D_{\bx \bx}^2 \f(t,\bx)\si(\cdot, \tilde \bx^t, \bx^\cdot, a)\big\rangle_H + F(t, \bx, a) \bigg\} \ge 0.
\]

$(ii)$ $u$ is said to be a viscosity sub-solution of \eqref{DPE} if $u(T, \cdot) \le G$ and, for all $\f \in C^{1,2}([0,T] \times H)$ such that $u-\f$ has a local maximum in $(t,\bx)$, we have
\[
-\pa_t \f(t,\bx) -  \sup_{a \in A}\bigg\{ \big\langle D_\bx \f(t, \bx), b_t(\cdot, \tilde \bx^t, \bx^\cdot, a) \big\rangle_H + \frac 1 2 \big\langle \si_t(\cdot, \tilde \bx^t, \bx^\cdot, a), D_{\bx \bx}^2 \f(t,\bx)\si(\cdot, \tilde \bx^t, \bx^\cdot, a)\big\rangle_H + F(t, \bx, a) \bigg\} \le 0.
\]

$(iii)$ $u$ is said to be a viscosity solution of \eqref{DPE} if it is both a viscosity super-solution and viscosity sub-solution of \eqref{DPE}. 
\end{definition}

Then, applying standard viscosity theory on Hilbert space (see e.g.\ \cite[Theorem 3.67]{fabbri2017stochastic}), we may formulate the following characterisation of $\cV$.
\begin{proposition}\label{prop:existence-uniqueness}
Assume that

\medskip
$(i)$ $t \longmapsto  B_t(\bx, a)$ is continuous, uniformly in $(\bx, a) \in B_H(0,R) \times A$ for all $R > 0;$

\medskip
$(ii)$ $\si$ has linear growth in $\tilde x^t$ and $\tilde x^s$, uniformly in the other variables. 

\medskip
Then $\cV$ is the unique continuous viscosity solution of \eqref{DPE} with polynomial growth. 
\end{proposition}

\begin{proof}
Let $(e_k)_{k \in \mathbb{N}^\smalltext{\star}}$ be an orthonormal basis of $H$. We essentially have to check \cite[Assumption (3.155)]{fabbri2017stochastic}, that is
\begin{equation}\label{3.155}
\underset{N \to \infty}{\lim} \sup_{a \in A}\Big\{ {\rm Tr}\big[\Si_t(\bx, a)\Si_t(\bx,a)^\top \cQ_N\big]\Big\} = 0, \; \forall(t, \bx) \in [0,T] \times H,
\end{equation}
where $\cQ_N$ is the orthonormal projection onto the family $(e_k)_{k \in\mathbb{N}^\smalltext{\star}\setminus\{1,\dots,N\}}$. Note that $\Si_t(\bx, a)\Si_t(\bx,a)^\top$ corresponds to the endomorphism of $H$
\[
 \by \longmapsto  \langle \si_t(\cdot, \tilde x^t, \tilde x^\cdot, a), \by \rangle_H \si_t(\cdot, \tilde x^t, \tilde x^\cdot, a). 
 \]
Then, denoting by $\si^k_t(\bx, a)$ the projection of $\Si_t(\bx, a)$ onto $e_k$, for $k\in\mathbb{N}^\star$, we have
\[
 {\rm Tr}\big[\Si_t(\bx, a)\Si_t(\bx,a)^\top \cQ_N\big] = \sum_{k = N+1}^\infty | \si^k_t(\bx, a) |^2. 
\]
However, we have 
\[
| \si^k_t(\bx, a) | = \lvert \langle \si_t(\cdot, \tilde x^t, \tilde x^\cdot, a), e_k \rangle_H \rvert \le C(1 + | \tilde x^t |) | \langle 1, e_k \rangle_H | + | \langle \bx, e_k \rangle_H |,
\]
and therefore
\[
 {\rm Tr}\big[\Si_t(\bx, a)\Si_t(\bx,a)^\top \cQ_N\big] \le 2(1+| \tilde x^t |)^2 \sum_{k = N+1}^\infty | \langle 1, e_k \rangle_H |^2 + 2\sum_{k = N+1}^\infty | \langle \bx, e_k \rangle_H |^2.  
 \]
Since both $1$ (as a constant mapping) and $\bx$ belong to $H$, the two sums on the right-hand side go to $0$ as $N \longrightarrow \infty$. Since this term is independent from $a$, we deduce that \eqref{3.155} holds true, and we may therefore conclude by applying \cite[Theorem 3.67]{fabbri2017stochastic}.
\end{proof}

\subsection{The case of uncontrolled volatility}\label{sect:BSDE}

In this section, we assume that $\si$ does not depend on $a$, and that there exists a bounded $\th : [0,T] \times H \times A \longrightarrow \dbR$ such that $B_t(\bx, a) = \G_t(\bx) + \Si_t(\bx)\th_t(\bx,a)$. We show that the value function of the infinite dimensional control problem can be expressed as the solution of a backward SDE. To this end, we reformulate the lifted control problem in weak formulation. Let $\bX$ be the unique strong solution of the $H$-valued SDE
\[
 \bX_t = \bx_0 + \int_0^t \G_r(\bX_r)\mathrm{d}r + \int_0^t \Si_r(\bX_r)\mathrm{d}W_r, \; \mbox{$\dbP$--a.s.} 
 \]
Let $\a \in \cA$. By the existence of the function $\th$ introduced above, it follows from Girsanov's theorem that there exists a probability measure $\dbP^\a$ equivalent to $\dbP$ such that
\[
 W_t^\a \coloneqq  W_t - \int_0^t \th_r(\bX_r, \a_r)\mathrm{d}r,
 \]
is a $\dbP^\a$--Brownian motion. Therefore
\[
 \bX_t = \bx_0 + \int_0^t B_r(\bX_r, \a_r)dr + \int_0^t \Si_r(\bX_r)dW_r^\a, \;\mbox{$\dbP^\a$--a.s.}, 
 \]
and we may reformulate the control problem in the following way
\begin{align}\label{weak-formulation}
\cV_0^w(\bx_0) = \sup_{\a \in \cA} \dbE^{\dbP^\smalltext{\a}}\bigg[ \int_0^T F_r(\bX_r, \a_r)\mathrm{d}r + G(\bX_T) \bigg].
\end{align}
We can easily see that an analogue of \Cref{prop:aux-to-or} holds true here; indeed, if $X$ writes 
\[
X_t = x_0 + \int_0^t \g_r(t, X_r)\mathrm{d}r + \int_0^t \si_r(t, X_r)\mathrm{d}W_r, \; \mbox{$\dbP$--a.s.,}
\]
and 
\[
V_0(x_0) \coloneqq  \sup_{\a \in \cA} \dbE^{\dbP^\smalltext{\a}}\bigg[\int_0^T f_r(X_r, \a_r)\mathrm{d}r + g(X_T)\bigg],
\]
then $V_0(x_0) = \cV_0^w(\bp(x_0))$.  Let now $\cH : [0,T] \times H \times \dbR \longrightarrow \dbR$ be the Hamiltonian defined by 
\[
\cH_t(\bx, z) = \sup_{a \in A}\big\{ z\th_t(\bx,a) + F_t(\bx, a)\big\}.
\] 
\begin{proposition}\label{prop:BSDE}
Assume that

\medskip
 $(i)$ $\cH$ is Lipschitz-continuous in $z;$
 
 \medskip
 $(ii)$ $F$ has linear growth in $\bx \in H$, uniformly in the other variables$;$
 
 \medskip
 $(iii)$ $\G$ and $\Si$ are Lipschitz-continuous in $\bx \in H$, uniformly in $t \in [0,T]$.
 
 \medskip
 Then $\cV_0^w(\bx_0) = Y_0$, where $(Y, Z)$ is the unique solution of the backward {\rm SDE}
\[
Y_t = G(\bX_T) + \int_t^T \cH_r(\bX_r, Z_r)\mathrm{d}r - \int_t^T Z_r \mathrm{d}W_r,\; t\in[0,T],\; \mathbb{P}\text{\rm--a.s.} 
\]
Furthermore, if there exists a measurable mapping $\psi : [0,T] \times H \times \dbR \longrightarrow A$ such that 
\begin{equation}\label{optimal-equality}
 \cH(t,\bx, z) = z\th_t\big(\bx, \psi_t(\bx, z)\big) + F_t\big(\bx, \psi_t(\bx, z)\big),
 \end{equation}
then $\a_t^* \coloneqq  \psi_t(\bX_t, Z_t)$ is an optimal control for \eqref{weak-formulation}. 
\end{proposition}
\begin{proof}
For $\a \in \cA$, let $(Y^\a, Z^\a)$ denote the solution of the backward SDE:
\[
 Y_t^\a = G(\bX_T) + \int_t^T \big(F_r(\bX_r, \a_r) + Z_r^\a \th_r(\bX_r, \a_r)\big)\mathrm{d}r - \int_t^T Z_r^\a \mathrm{d}W_r.  
 \]
By the Lipschitz and linear growth assumptions made on $\cH$, $\G$, $\Si$ and $F$ and the boundedness of $\th$, there exists a unique solution $(Y^\a, Z^\a)$ to the above equation (see e.g.\ \citeauthor*{pardoux1990adapted} \cite{pardoux1990adapted}, \citeauthor*{el1997backward} \cite{el1997backward} or \citeauthor*{zhang2017backward} \cite{zhang2017backward}). As the equations solved by $(Y, Z)$ and $(Y^\a, Z^\a)$ satisfy the usual Lipschitz and measurability conditions, and by definition of $\cH$, the comparison principle for backward SDEs ensures that $Y_0^\a \le Y_0$. Since $\a$ is arbitrary, this shows that $\cV_0^w(\bx_0) \le Y_0$. 

\medskip
Fix now $\e > 0$. By measurable selection, there exists a measurable mapping $\psi^\e : [0,T] \times H \times \dbR \longrightarrow \dbR$ such that
\[
 \cH_t(\bx, z) \le F_t(\bx, \psi^\e_t(\bx, z)) + z\th_t(\bx, \psi^\e_t(\bx, z)) + \e. 
\]
Introducing $\a_t^\e \coloneqq  \psi^\e_t(\bX_t, Z_t)$, we have
\begin{align*}
Y_t - Y_t^\e &= \int_t^T \big(\cH_r(\bX_r, Z_r) - F_r(\bX_r, \a_r^\e) - Z_r^{\a^\smalltext{\e}} \th_r(\bX_r, \a_r^\e)\big)\mathrm{d}r - \int_t^T \big(Z_r - Z_r^{\a^\smalltext{\e}})\mathrm{d}W_r \\
&= \int_t^T \Big(\big(\cH_r(\bX_r, Z_r) - F_r(\bX_r, \a_r^\e) - Z_r \th_r(\bX_r, \a_r^\e)\big) + (Z_r - Z_r^{\a^\smalltext{\e}}) \th_r(\bX_r, \a_r^\e) \Big)\mathrm{d}r - \int_t^T \big(Z_r - Z_r^{\a^\smalltext{\e}})\mathrm{d}W_r \\
&= \int_t^T\big(\cH_r(\bX_r, Z_r) - F_r(\bX_r, \a_r^\e) - Z_r \th_r(\bX_r, \a_r^\e) \big)\mathrm{d}r - \int_t^T \big(Z_r - Z_r^{\a^\smalltext{\e}})\mathrm{d}W_r^{\a^\smalltext{\e}} \\
&\le \e(T-t) - \int_t^T \big(Z_r - Z_r^{\a^\smalltext{\e}})\mathrm{d}W_r^{\a^\smalltext{\e}}.
\end{align*}
Thus, we have $Y_0 \le Y_0^{\a^\smalltext{\e}} + T\e$. By arbitrariness of $\e$, this implies that $Y_0 \le \cV_0^w(\bx_0)$, and therefore the desired equality holds true. In particular, when \eqref{optimal-equality} holds, we have $\cV_0^w(\bx_0) = Y_0 = Y_0^{\a^\smalltext{\star}}$, which means that $\a^\star$ is an optimal control. 
\end{proof}

\begin{remark}{
Let us discuss what the assumption $B_t(\bx, a) =\G_t(\bx) + \Si_t(\bx)\th_t(\bx,a)$ means in the context of the control of stochastic Volterra integral equations. If one wants to be able to apply Girsanov's theorem, the real-valued mapping $\th$ must depend only on $\bx$ and on the \enquote{regular time} $t$, and not on the \enquote{Volterra time} $s$. This means that the dependence on $s$ must be the same in $b$ and $\si$. This is for instance the case for the following dynamics, considered by {\rm\citeauthor*{di2023lifting} \cite{di2023lifting}}
\begin{equation}\label{Volterra-BSDE}
 X_t = x + \int_0^t K(t-r)\Big( \big(b^1_r(X_r) + \si_r(X_r)b^2_r(X_r, \a_r)\big)\mathrm{d}r + \si_r(X_r)\mathrm{d}W_r\Big).
 \end{equation}
This restriction is closely related to the difficulty of handling different kernels in the drift and in the volatility under a Girsanov transformation. Since the change of measure acts through the volatility coefficient, one needs the dependence on the Volterra parameter to factor in a compatible way in both terms. In particular, if the drift and the volatility involve genuinely different kernels, the present argument does not apply directly. 
}
\end{remark}

\section{Examples}\label{sec:examples}

\subsection{A---very---simple starter}

We start with a simple example to illustrate the lifting procedure on an elementary problem. Consider the uncontrolled SDE
\[
\mathrm{d}X_r = X_r\mathrm{d}r + \mathrm{d}W_r.
\]
The corresponding lifted dynamics, started at time $t$ from some $\bx\in H$, is the family $\bX^{t,\bx}\coloneqq (X^{t,\bx,s})_{0\le s\le T}$ defined by
\[
X_r^{t,\bx,s} = \bx^s + \int_t^r X_u^{t,\bx,u}\mathrm{d}u + W_r-W_t,\; r\in[t,T],\; s\in[0,T].
\]

Set $u(t,\bx)\coloneqq \dbE^\P\big[X_T^{t,\bx,T}\big].$ If $m(r)\coloneqq \dbE^\P[X_r^{t,\bx,r}]$, then
\[
m(r)=\bx^r+\int_t^r m(u)\mathrm{d}u.
\]
Writing $y(r)\coloneqq \int_t^r m(u)\mathrm{d}u$, we obtain the ODE
\[
y^\prime(r)=\bx^r+y(r),\; y(t)=0,
\]
and therefore
\[
y(r)=\mathrm e^{r}\int_t^r \mathrm e^{-u}\bx^u\mathrm{d}u.
\]
Hence
\[
u(t,\bx)=m(T)=\bx^T+\mathrm e^{T}\int_t^T \mathrm e^{-s}\bx^s\mathrm{d}s.
\]

As a continuous linear functional on $H$
\[
D_\bx u(t,\bx)\cdot h = h^T+\mathrm e^{T}\int_t^T \mathrm e^{-s}h^s\mathrm{d}s,\; h\in H,\; \text{\rm and}\; \pa_t u(t,\bx)=-\mathrm e^{T-t}\bx^t.
\]
Since the drift of the lifted dynamics is the constant function $B(t,\bx)(s)=\bx^t$, we obtain
\[
D_\bx u(t,\bx)\cdot B(t,\bx) = \bx^t+\mathrm e^{T}\int_t^T \mathrm e^{-s}\bx^t\mathrm{d}s = \mathrm e^{T-t}\bx^t = -\pa_t u(t,\bx).
\]
Moreover, $D_{\bx\bx}^2u\equiv0$. Hence $u$ solves the corresponding linear equation in the lifted space.

\subsection{Linear--quadratic control problem with kernel}

Let $\phi : [0,T] \longrightarrow \dbR$ be continuous, and set
\[
k_t(s)\coloneqq \1_{(t,T]}(s)\phi(s-t) + \1_{[0,t]}(s)\phi(0), \; (t,s)\in[0,T]^2.
\]
We consider the controlled Volterra-type dynamics
\[
 X_t = x_0 + \int_0^t \phi(t-s)\big( (X_s + \a_s)\mathrm{d}s + \mathrm{d}W_s\big), 
\]
and the control problem
\[
V_0 \coloneqq  \sup_{\alpha \in \cA} \dbE^{\P^\smallertext{\alpha}}\bigg[ - \frac 1 2 \int_0^T (X_s^2 + \a_s^2)\mathrm{d}s\bigg]. 
\]

Notice that there is no conceptual obstruction to adding a terminal reward in this example; it would merely modify the terminal condition in the Riccati system, at the price of heavier notation. The corresponding lifted problem writes
\begin{equation}\label{liftedLQ}
 V(t, \bx) \coloneqq   \sup_{\alpha \in \cA} \dbE^{\P^\smallertext{\alpha}}\bigg[ - \frac 1 2 \int_t^T \big((X_s^{t, \bx, s})^2 + \a_s^2\big)\mathrm{d}s\bigg], 
 \end{equation}
where the flow $\bX^{t, \bx} \coloneqq  ( X^{t,\bx,s})_{0 \le s \le T}$ is such that
\[
X_r^{t, \bx, s} = \bx^s + \int_t^{r} k_\th(s)\big( X_\th^{t, \bx, \th} + \a_\th \big) \mathrm{d}\th + \int_t^{r}k_\th(s) \mathrm{d}W_\th,
\qquad s\in[0,T].
\]

\medskip
We easily see that the dynamic programming equation corresponding to this problem is
\begin{align}\label{DPE-LQ}
\nonumber&-\pa_t u(t,\bx) -  \frac 1 2 \iint_{[0,T]^\smalltext{2}} D_{\bx\bx}^2 u(t,\bx)(r,s)k_t(r)k_t(s)\mathrm{d}r\mathrm{d}s + x^t \int_0^T D_\bx u(t, \bx)(r) k_t(r)\mathrm{d}r\\
&\quad + \frac 1 2 \bigg(\int_0^T D_\bx u(t, \bx)(r) k_t(r) \mathrm{d}r \bigg)^2 + \frac{(x^t)^2}{2} = 0, 
\end{align}
with boundary condition $u(T, \cdot) = 0$. Here we use Lebesgue measure instead of $\m$ as reference measure and $\dbL^2([0,T], \mathrm{d}t)$ as reference space to represent the Fréchet derivatives of $u$ (which does not make a difference in the context of classical solutions). 

\medskip
Our objective is to find a solution $u$ of the form
\[
u(t, \bx) =  \frac 1 2 \iint_{[0,T]^\smalltext{2}} c(t,r,s) x^r x^s \mathrm{d}r\mathrm{d}s, 
\]
where $c$ is a measurable $\dbR$-valued function defined on $[0,T]^3$, which is symmetric in its last two variables. Let us compute formally the derivatives of $u$
\begin{align*}
&\pa_t u(t,\bx) =  \frac 1 2 \iint_{[0,T]^\smalltext{2}} \pa_t c(t,r,s) x^r x^s \mathrm{d}r\mathrm{d}s, \; D_\bx u(t,\bx)(r) =  \int_0^T c(t,r,s)x^s\mathrm{d}s, \; D_{\bx \bx}^2 u(t,\bx)(r,s) = c(t,r,s).
\end{align*}
Observing that $x^t = \int_0^T \d_t(s)x^s\mathrm{d}s$, where $\d_t$ is the Dirac mass at $t$, we also compute
\begin{gather*}
x^t \int_0^T D_\bx u(t, \bx)(r) k_t(r)\mathrm{d}r =  \iint_{[0,T]^\smalltext{2}} c(t,r,s)k_t(r)\d_t(s) x^r x^s \mathrm{d}r\mathrm{d}s, \\
(x^t)^2 = \iint_{[0,T]^\smalltext{2}} \d_t(r)\d_t(s)x^r x^s \mathrm{d}r\mathrm{d}s, 
\end{gather*}
and
\begin{align*}
\bigg(\int_0^T D_\bx u(t, \bx)(r) k_t(r)\mathrm{d}r \bigg)^2 &=  \iint_{[0,T]^\smalltext{2}} \bigg(\int_0^T c(t,r,\th)k_t(\th)\mathrm{d}\th \bigg)\bigg(\int_0^T c(t,\t,s)k_t(\tau)\mathrm{d}\t \bigg)x^r x^s \mathrm{d}r \mathrm{d}s.
\end{align*}
Introduce the notation
\[
(c \star \phi)(t,r) \coloneqq  \int_0^T c(t,r,\th)k_t(\th)\mathrm{d}\th. 
\]
Note that, since $c$ is symmetric in $r$ and $s$, we also have 
$(c \star \phi)(t,s) = \int_0^T c(t,\t,s)k_t(\t)\mathrm{d}\t$. Plugging all these expressions into the dynamic programming equation \eqref{DPE-LQ}, we see that $c$ satisfies the following equation
\begin{equation}\label{first-riccati}
\pa_t c(t,r,s) =  -(c\star\phi)(t,r)(c\star\phi)(t,s) - 2 \d_t(s)(c\star\phi)(t,r) - \d_t(r)\d_t(s),
\end{equation}
with terminal condition $c(T, \cdot, \cdot) = 0$. 


\begin{remark}
The above verification extends to more general linear--quadratic kernels, similarly to {\rm\citeauthor*{wang2023linear} \cite{wang2023linear}}. In particular, \eqref{first-riccati} corresponds to $(4.9)$ in {\rm\cite{wang2023linear}} for our choice of coefficients, and is closely related to {\rm Equation $(3.2)$} in {\rm\citeauthor*{abi2021linear} \cite{abi2021linear}}, where a Riccati system is derived for a special kernel by a different method.
\end{remark}
The next two sections spell out in detail two genuinely Volterra applications of the same mechanism: a regular propagator model for optimal liquidation and a stochastic advertising/goodwill model with carryover or distributed forgetting.

\subsubsection{Optimal liquidation with transient price impact under a regular propagator}\label{subsec:propagator-example}

A natural genuinely Volterra control problem covered by our framework is a regular-kernel version of optimal liquidation with transient price impact. This class of models originates in {\rm\citeauthor*{gatheral2012transient} \cite{gatheral2012transient}}, was extended to signal--adaptive trading with exponential propagators by {\rm\citeauthor*{neuman2022optimal} \cite{neuman2022optimal}}, and was recently solved in full generality for Volterra propagators by {\rm\citeauthor*{abijaber2025optimal} \cite{abijaber2025optimal}}.

\paragraph*{Model.}
Fix a bounded control set $A=[0,\bar \n]$, an initial inventory $q_0\in\dbR$, a temporary impact parameter $\l>0$, non-negative inventory penalties $\phi_\smallertext{\rm inv}$, $\varrho$, and an unaffected price process $P$ solving
\[
P_t = p_0 + \int_0^t \m_r(P_r)\mathrm{d}r + \int_0^t \si_r(P_r)\mathrm{d}W_r.
\]
Let $K:[0,T]^2\longrightarrow \dbR$ be continuous, satisfy $K(t,r)={K(t,t)}$ for $r>t$, and assume that for each fixed $r$ the map $t\longmapsto K(t,r)$ has the Sobolev regularity required in \Cref{prop:SDE}. For a trading speed $\n\in\cA$, define the inventory and impact states
\[
Q_t^\n = q_0 - \int_0^t \n_r \mathrm{d}r,
\;
I_t^\n = \int_0^t K(t,r)\n_r\mathrm{d}r.
\]
The execution price is $S_t^\n = P_t - \l \n_t - I_t^\n,$ and the liquidation problem consists in maximising
\begin{equation}\label{eq:propagator-objective}
J(\n)\coloneqq \dbE^{\mathbb P}\bigg[\int_0^T S_t^\n \n_t\mathrm{d}t + P_TQ_T^\n - \phi_\smallertext{\rm inv}\int_0^T (Q_t^\n)^2\mathrm{d}t - \varrho(Q_T^\n)^2\bigg].
\end{equation}
The hard liquidation constraint $Q_T^\n=0$ often imposed in execution models may either be approximated within our framework by taking a large terminal penalty $\varrho$, or treated through the state-constraint/stochastic-target perspective discussed later in this section.

\paragraph*{Fit with our framework.}
Set $X^\n\coloneqq(Q^\n,P,I^\n)\in\dbR^3$. Then \eqref{eq:propagator-objective} is exactly of the form \eqref{VolterraSDE}--\eqref{control-pb}, with coefficients
\[
b_r\big(t,(q,p,i),a\big)
=
\begin{pmatrix}
-a\\[0.2em]
\m_r(p)\\[0.2em]
K(t,r)a
\end{pmatrix},
\;
\si_r\big(t,(q,p,i),a\big)
=
\begin{pmatrix}
0\\[0.2em]
\si_r(p)\\[0.2em]
0
\end{pmatrix},
\]
and rewards
\[
f\big(t,(q,p,i),a\big)= (p-\l a-i)a - \phi_\smallertext{\rm inv} q^2,\; g(q,p,i)=pq-\varrho q^2.
\]
The Volterra parameter only enters through the propagator $K(t,r)$, so \Cref{prop:SDE,prop:aux-to-or,prop:existence-uniqueness} apply as soon as $t\longmapsto K(t,r)$ satisfies the Sobolev regularity assumed in Section~2.

\paragraph*{Lifted equation and feedback form.}
Writing $\bx=(\bx_1,\bx_2,\bx_3)\in H=(W^{1,2}([0,T]))^3$, the lifted coefficients are
\[
B(r,\bx,a)(s)=
\begin{pmatrix}
-a\\[0.2em]
\m_r(\bx_2^r)\\[0.2em]
K(s,r)a
\end{pmatrix},
\;\Si(r,\bx,a)(s)=
\begin{pmatrix}
0\\[0.2em]
\si_r(\bx_2^r)\\[0.2em]
0
\end{pmatrix},
\]
for $(r,s)\in[0,T]^2$. Hence the value function is characterised by the lifted HJB equation
\begin{align*}
-\pa_t u(t,\bx)
-\sup_{a\in A}\bigg\{
&\langle D_\bx u(t,\bx),B(t,\bx,a)\rangle_H
+\frac12\big\langle \Si(t,\bx,a),D_{\bx\bx}^2u(t,\bx)\Si(t,\bx,a)\big\rangle_H +(\bx_2^t-\l a-\bx_3^t)a-\phi_\smallertext{\rm inv}(\bx_1^t)^2
\bigg\}=0,
\end{align*}
with terminal condition
\[
u(T,\bx)=\bx_2^T\bx_1^T-\varrho(\bx_1^T)^2.
\]
If $u$ is smooth, define the two elements of $H$ given by
\[
e_Q(s)\coloneqq
\begin{pmatrix}
1\\
0\\
0
\end{pmatrix},
\;
K_t^\smallertext{\rm imp}(s)\coloneqq
\begin{pmatrix}
0\\
0\\
K(s,t)
\end{pmatrix},
\; s\in[0,T].
\]
Then the Hamiltonian is quadratic in $a$, and the interior first-order condition gives the feedback
\[
a^\star(t,\bx)=\Pi_A\bigg(\frac{\bx_2^t-\bx_3^t-\langle D_\bx u(t,\bx),e_Q\rangle_H+\langle D_\bx u(t,\bx),K_t^\smallertext{\rm imp}\rangle_H}{2\l}\bigg),
\]
where $\Pi_A$ denotes the Euclidean projection onto $A$. When $K\equiv0$, the model collapses to a standard Markovian liquidation problem. When $K(t,r)=\eta \mathrm{e}^{-\rho(t-r)}\1_{\{r\le t\}} {+ \eta \1_{\{r > t \}}}$, the transient impact becomes one-dimensional and one recovers the exponential-propagator setting of {\rm\cite{neuman2022optimal}}. For a genuinely non-exponential regular kernel, the state remains non-Markovian in finite dimension but becomes Markov after the Sobolev lift constructed in \Cref{sec:infinite-dimension,sec:DPE}.

\subsubsection{Advertising and goodwill with carryover or distributed forgetting}\label{subsec:advertising-example}

A second self-contained application outside finance is a stochastic advertising/goodwill model with carryover or distributed forgetting. Memory effects of this kind are classical in advertising theory and have been treated in delay form by {\rm\citeauthor*{gozzi2009controlled} \cite{gozzi2009controlled}} and, in a Volterra setting, by {\rm\citeauthor*{giordano2024optimal} \cite{giordano2024optimal}} and {\rm\citeauthor*{gozzi2024optimal} \cite{gozzi2024optimal}}.

\paragraph*{Model.}
Let $A=[0,\bar a]$ and consider an advertising rate $a\in\cA$. We model the goodwill stock by
\begin{equation}\label{eq:advertising-goodwill}
Y_t^a = y_0 + \int_0^t K(t-r)\big(\b_\smallertext{\rm ad}a_r-\d_\smallertext{\rm ad}Y_r^a\big)\mathrm{d}r + \si_\smallertext{\rm ad}\int_0^t K(t-r)\mathrm{d}W_r,
\end{equation}
where the coefficients $\b_\smallertext{\rm ad}$, $\d_\smallertext{\rm ad}$,and $\si_\smallertext{\rm ad}$ are positive, and the kernel $K:[0,T]\longrightarrow\dbR$ is such that the functions
\[
K_t(s)\coloneqq \1_{[t,T]}(s)K(s-t) {+ K(0)\1_{[0,t)}(s)},\; (t,s)\in[0,T]^2,
\]
have the Sobolev regularity required by \Cref{prop:SDE}. A natural performance criterion is
\begin{equation}\label{eq:advertising-objective}
J(a)\coloneqq \dbE^{\mathbb P}\bigg[\int_0^T \bigg(\eta Y_t^a-\frac{\L}{2}a_t^2\bigg)\mathrm{d}t+\G Y_T^a\bigg],
\end{equation}
where $\eta,\L,\G\ge0$ quantify, respectively, the running value of goodwill, the cost of advertising effort, and the terminal value of the product's reputation stock.

\paragraph*{Fit with our framework.}
The pair \eqref{eq:advertising-goodwill}--\eqref{eq:advertising-objective} is of the form \eqref{VolterraSDE}--\eqref{control-pb}, with
\[
b_r(t,x,a)=K(t-r)\big(\b_\smallertext{\rm ad}a-\d_\smallertext{\rm ad}x\big),
\;
\si_r(t,x,a)=\si_\smallertext{\rm ad}K(t-r),\; f(t,x,a)=\eta x-\frac{\L}{2}a^2,
\; g(x)=\G x.
\]
Thus the problem is genuinely Volterra whenever $K$ is neither a Dirac mass nor an exponential kernel giving rise to a finite-dimensional Markov reduction.

\paragraph*{Lifted equation and feedback form.}
For $\bx\in H$, the lifted coefficients are
\[
B(r,\bx,a)(s)=(\1_{\{r\le s\}}K(s-r) {+ \1_{\{ r > s \}}K(0)})\big(\b_\smallertext{\rm ad}a-\d_\smallertext{\rm ad}\bx^r\big)
=K_r(s)\big(\b_\smallertext{\rm ad}a-\d_\smallertext{\rm ad}\bx^r\big),
\]
and
\[
\Si(r,\bx,a)(s)=(\si_\smallertext{\rm ad}\1_{\{r\le s\}}K(s-r){ + \1_{\{ r > s \}}K(0)})=\si_\smallertext{\rm ad}K_r(s).
\]
Hence the value function solves
\begin{align*}
-\pa_t u(t,\bx)
-\sup_{a\in A}\bigg\{
&\big(\b_\smallertext{\rm ad}a-\d_\smallertext{\rm ad}\bx^t\big)\langle D_\bx u(t,\bx),K_t\rangle_H
+\frac{\si_\smallertext{\rm ad}^2}{2}\big\langle K_t,D_{\bx\bx}^2u(t,\bx)K_t\big\rangle_H +\eta \bx^t-\frac{\L}{2}a^2
\bigg\}=0,
\end{align*}
with terminal condition
\[
u(T,\bx)=\G\bx^T.
\]
If $u$ is smooth, the maximiser is explicitly
\[
a^\star(t,\bx)=\Pi_A\bigg(\frac{\b_\smallertext{\rm ad}}{\L}\langle D_\bx u(t,\bx),K_t\rangle_H\bigg).
\]
In particular, our method provides a direct HJB characterisation for a memory-goodwill model in which the entire carryover profile enters through the Volterra kernel. In the special case $\eta=0$ and $\G>0$, one recovers the terminal-goodwill objective considered in {\rm\cite{giordano2024optimal}}; quadratic running penalties in the state lead back to the Riccati structure derived above.

\begin{remark}
The same regular-kernel mechanism also appears in reduced-order hereditary engineering models. For instance, Galerkin truncations of stochastic heat equations with memory lead to finite-dimensional systems of the form
\[
X_t = x_0 + \int_0^t K(t-r)\big(AX_r+Bu_r\big)\mathrm{d}r + \int_0^t K(t-r)\G_0\mathrm{d}W_r,
\]
which fit the standing assumptions of {\rm \Cref{sec:infinite-dimension,sec:DPE}} as soon as $t\longmapsto K(t,r)$ is regular enough in the Volterra variable; see {\rm\citeauthor*{confortola2014optimal} \cite{confortola2014optimal}}. The full {\rm PDE} problem is of course infinite-dimensional already before the lift in the memory variable, so a complete treatment would require combining the present approach with a spatial infinite-dimensional analysis.
\end{remark}

\subsection{Time-inconsistent contract theory}

One of our main motivation is related to the works of \citeauthor*{hernandez2024time} \cite{hernandez2023me, hernandez2024time}, which focuses on principal--agent contracting problems in presence of a form of time-inconsistency in the agent's problem. We first recall the setting of the problem and the main results of \cite{hernandez2023me, hernandez2024time} when the time-inconsistency is due to the presence of a non-exponential discount factor.

\subsubsection{The agent's and principal's problems}

Given an output process 
\[ X_t^\a \coloneqq  X_0 + \int_0^t \a_r \mathrm{d}r + W_t, \]
where the effort $\a$ takes its value in some compact $[0, \bar a] \subset \dbR$, and a payment $\xi$ given by the principal, the agent wants to solve the control problem
\[ V_0^\smallertext{\rm A}(\xi) \coloneqq  \sup_{\a \in \cA} \dbE^{\P}\bigg[U_\smallertext{\rm A}(0,\xi) - \frac 1 2 \int_0^T c_r(0,\a_r)\mathrm{d}r \bigg],
\]
where $U_\smallertext{\rm A}$ corresponds to his utility function, $c$ his cost function and $f : \dbR \longrightarrow \dbR$ to the (possibly non-exponential) discount factor. 
The dynamic version of the agent's problem, \emph{i.e.} the control problem seen from any date $t \in [0,T]$, takes the following form
\[
 V_t^\smallertext{\rm A}(\xi) \coloneqq \sup_{\a \in \cA} \dbE^{\P}\bigg[ U_\smallertext{\rm A}(t,\xi) - \frac 1 2 \int_t^T c_r(t,\a_r)\mathrm{d}r \bigg]. 
 \]

Clearly, such a problem may not be handled through the traditional dynamic programming approach. Instead, the authors of \cite{hernandez2023me} assume that the agent plays a leader--follower game with the future versions of himself, therefore looking for a Stackelberg equilibrium, see \cite[Definition 2.6]{hernandez2023me}. In particular, they show that 
\[
 V_t^\smallertext{\rm A}(\xi) = Y_t^t, 
 \]
where the family of processes $(Y^s)_{\{s \in [0,T] \}}$ satisfies the backward system
\[
Y_t^s = U_\smallertext{\rm A}(s,\xi) - \int_t^T c_r^\star(s,Z_r^r)\mathrm{d}r - \int_t^T Z_r^s \mathrm{d}W_r, \; t\in[0,T],
\]
with $c_r^{\star}(z) \coloneqq c_r(a_r^{\star}(z))$, where $a^\star$ corresponds to a Stackelberg equilibrium. Assuming $c$ is continuous in both $r$ and $a$ and non-negative, we observe that, by boundedness of the controls $\a$, $c^\star$ takes its values in some compact $[0,\bar c]$.
In the spirit of \cite{cvitanic2018dynamic}, this system is rewritten in a forward way and the principal optimises on $\xi$ by maximising with respect to $Z$ and $Y_0$
\[
V^\smallertext{\rm P} = \sup_{\{\xi : V^\smalltext{\rm A}(\xi) \ge R\}} \dbE^{\P}\big[U_\smallertext{\rm P}(X_T, \xi)\big] = \sup_{Y_0^\cdot \ge R} V(Y_0^\cdot), 
\]
where $R$ is the participation constraint (\emph{i.e.} the minimal utility guaranteed to the agent so that he accepts to sign the contract), and $V$ is defined by 
\begin{equation}\label{principal}
 V(Y_0^\cdot) \coloneqq \sup_{Z \in \cZ} \dbE^\P\big[U_\smallertext{\rm P}(X_T^Z, Y_T^{0,Z})\big], 
 \end{equation}
where $Y^{s,Z}$ has the forward dynamics under $\P$
\begin{equation}\label{forward-dynamics}
Y_t^{s,Z} = Y_0^{s} + \int_0^t c_r^\star(s,Z_r^r)\mathrm{d}r + \int_0^t Z_r^s \mathrm{d}W_r, \; t\in[0,T],
\end{equation}
and $\cZ$ is the set of square integrable doubly indexed processes such that 
\begin{equation}\label{stochastic-target1} 
U_\smallertext{\rm A}^{(-1)}(s, Y_T^s) = U_\smallertext{\rm A}^{(-1)}(0, Y_T^0),\; \forall s \in [0,T], 
\end{equation}
where $U_\smallertext{\rm A}^{(-1)}$ denotes the inverse of $U_\smallertext{\rm A}$ with respect to the second variable. Our goal is to study the control problem \eqref{principal} by using the setting developed in the present paper. 

\subsubsection{Reformulation as a control problem with stochastic target constraints}

We assume that the control $\bZ \coloneqq (Z^s)_{\{s \in [0,T]\}}$ is such that $\bZ_t \in H$ for all $t \in [0,T]$.
It is clear that those conditions ensure that $\bY^\bZ \coloneqq  (Y^{Z,s})_{\{s \in [0,T]\}}$ takes its values in $H$ as well. We then rewrite the principal problem as
\begin{equation}\label{principal2}
 V(\by) = \sup_{\bZ \in \cZ_\smalltext{H}} \dbE^\P\big[U_\smallertext{\rm P}(X_T, Y_T^{0,\bZ})\big], 
\end{equation}
where $\by = \bY_0^{\bZ}$, $\cZ_H$ is the set of square integrable $H$-valued processes such that \eqref{stochastic-target1} is satisfied, or equivalently (abusing the notation and denoting a Sobolev function and its continuous representative the same way)
\[
Y_T^{0,\bZ}= \psi\big(s,Y_T^{s,\bZ}\big),\; \forall s \in [0,T], 
\]
with $\psi(s, y) \coloneqq  U_\smallertext{\rm A}\big(0, U_\smallertext{\rm A}^{(-1)}(s,y)\big)$ for all $(s, y) \in [0,T] \times \dbR$. Thus, the Principal must solve a stochastic control problem with stochastic target constraints on a Hilbert space. Note that this constraint is equivalent to
\begin{equation}\label{target2}
\ul g(\bY_T^\bZ) \le Y_T^{0,\bZ} \le \ol g(\bY_T^\bZ),
\end{equation}
with
\[
 \ul g(\bx) \coloneqq \min_{0 \le s \le T} \psi(s,\tilde x^s),
 \;
 \ol g(\bx) \coloneqq \max_{0 \le s \le T} \psi(s,\tilde x^s).
 \]

\paragraph*{Reachability set.}
The first step for the principal is to determine her reachability set, \emph{i.e.} the family of sets $\cV(t)$, $t\in[0,T]$, in which the state process $\bY^\bZ$ must lie at each time so that the target \eqref{target2} can still be reached. Following the original ideas of {\rm\citeauthor*{soner2002stochastic} \cite{soner2002stochastic}} in finite dimension, and more recently of {\rm\citeauthor*{bouchard2020quenched} \cite{bouchard2020quenched}} in the Wasserstein setting, it is natural to introduce
\[
w(t,\by)\coloneqq \inf\big\{ y \in \dbR : \ul g(\bY_T^{t,\by,\bZ}) \le \hat Y_T^{0,t,y,\bZ} \le \ol g(\bY_T^{t,\by,\bZ}),\; \mbox{for some} \; \bZ \big\},
\]
where $\hat Y^{0,t,y,\bZ}$ denotes the one-dimensional process started from $y$ at time $t$ and driven by the same control as $Y^{0,\bZ}$. Formally, one expects $w$ to be related to the geometric dynamic programming equation
\begin{equation}\label{pde:general}
 - \pa_t u + \sup_{\bz \in \cN(t, \by, D_\by u(t,\by))} \Big\{ - \langle c_t^{\star}(\cdot, z^t), D_\by u(t,\by) \rangle_H - \frac{1}{2}\langle \bz, D_{\by \by}^2 u(t,\by)\bz \rangle_H \Big\} = 0,
 \end{equation}
where
\[
\cN(t, \by, \bp) \coloneqq \big\{ \bz \in H : \langle \bz, \bp \rangle_H - z^0 = 0 \big\}.
\]
This equation is the natural infinite-dimensional analogue of the geometric PDE of {\rm\citeauthor*{soner2002dynamic} \cite{soner2002dynamic}}. In \Cref{sec:appen} we prove such a statement for epigraph-type stochastic target problems on Hilbert spaces. Deriving the fully coupled equation \eqref{pde:general} in the present context is substantially more delicate and lies beyond the scope of this paper. For this reason, we introduce the two auxiliary epigraph-type target problems
\begin{align*}
\ul w(t, \by) &\coloneqq \inf\big\{ y \in \dbR : \hat Y_T^{0,t,y,\bZ} \ge \ul g(\bY_T^{t,\by,\bZ}),\; \mbox{for some} \; \bZ \big\}, \\
\ol w(t, \by) &\coloneqq \sup\big\{ y \in \dbR : \hat Y_T^{0,t,y,\bZ} \le \ol g(\bY_T^{t,\by,\bZ}),\; \mbox{for some} \; \bZ \big\},
\end{align*}
for $(t,\by)\in[0,T]\times H$.

\medskip
As studied in \Cref{sec:appen} (note that the supremum problem can be rewritten as an infimum problem), the corresponding formal equations on $[0,T]\times H$ are
\begin{gather*}
- \pa_t \ul w(t, \by) + \sup_{\bz \in \cN(t, \by, \ul w(t,\by), D_\by \ul w(t,\by))} \bigg\{ c_t^\star(0, z^t) - \langle c_t^\star(\cdot, z^t), D_\by \ul w(t,\by) \rangle_H + \frac{1}{2} \langle \bz, D_{\by\by}^2 \ul w(t, \by)\bz \rangle_H \bigg\} = 0,\; \ul w(T,\by) = \ul g(\by), \\
- \pa_t \ol w(t, \by) + \inf_{\bz \in \cN(t, \by, \ol w(t,\by), D_\by \ol w(t,\by))} \bigg\{ c_t^\star(0, z^t) - \langle c_t^{\star}(\cdot, z^t), D_\by \ol w(t,\by) \rangle_H + \frac{1}{2} \langle \bz, D_{\by\by}^2 \ol w(t, \by)\bz \rangle_H \bigg\} = 0,\; \ol w(T,\by) = \ol g(\by),
\end{gather*}
with
\[
 \cN(t, \by, y, \bp) \coloneqq \big\{ \bZ \in H : Z^0 - \langle \bZ, \bp \rangle_H = 0 \big\} = \big\{ \bZ \in H : \langle \bZ, \bv_0 - \bp \rangle_H = 0 \big\},
 \]
where $\bv_0$ is the element of $H$ such that $\langle \bv_0, \bx \rangle_H = x^0$ for all $\bx \in H$.

\medskip
We now informally describe how we expect the reachability set to be related to these two functions. First, the following inclusion is immediate:
\[
 \cV(t) \subset \big\{ \by \in H : \ul w(t, \by) \le y^0 \le \ol w(t, \by)\big\}.
 \]
Indeed, if $\by \in \cV(t)$, then there exists $\bZ$ such that
\[
\ul g(\bY_T^{t, \by, \bZ}) \le Y_T^{0, t, y^0, \bZ} \le \ol g(\bY_T^{t, \by, \bZ}).
\]
Since the family $(Y^{s,t,\by,\bZ})_{s\in[0,T]}$ only interacts through the control $\bZ$, we have $Y^{0,t,\by,\bZ}=\hat Y^{0,t,y^0,\bZ}$, $\dbP$--a.s. Therefore the same control is admissible for both auxiliary target problems, which implies that $y^0 \in [\ul w(t,\by),\ol w(t,\by)]$. The converse inclusion is more delicate; for the finite-dimensional analogue we refer to {\rm\citeauthor*{hernandez2024closed} \cite[Lemma 5.3]{hernandez2024closed}}.

\medskip
Assuming the same argument can be carried out here, we shall consider in what follows that the closure of the principal's reachability set is given by
\begin{equation}\label{reachability-assumption}
 {\rm cl}\big(\cV(t)) =  \big\{ \by \in H : \ul w(t, \by) \le y^0 \le \ol w(t, \by) \big\}.
 \end{equation}

\paragraph*{Dynamic programming equation with epigraph-type reachability set.}
Given this characterisation of the reachability set, the principal's problem may be reformulated as a state-constrained control problem
\[
 V(t, x,\by) \coloneqq \sup_{\{\bZ : \bY_r^{t,\by,\bZ} \in \cV(r),\ \forall r \in [t,T]\}} \dbE^\P\big[U_\smallertext{\rm P}(X_T^{t,x,\bZ}, Y_T^{0,\bZ})\big],
 \]
with $(X^{t, x, \bZ}, \bY^{t, \by, \bZ})$ following the dynamics
\begin{gather*}
X_s^{t, x, \bZ} = x + \int_t^s a_r^\star(Z_r^r)\mathrm{d}r + W_s - W_t,\qquad
\bY_s^{t,\by,\bZ} = \by + \int_t^s c_r^\star(\cdot, Z_r^r)\mathrm{d}r + \int_t^s \bZ_r  \mathrm{d}W_r,
\quad s\in[t,T].
\end{gather*}
Introduce the sets
\[
\Oc(t)\coloneqq \big\{\by\in H: \ul w(t, \by) < y^0 < \ol w(t,\by)\big\},\quad \underline{\Oc}(t)\coloneqq \big\{\by\in H: \ul w(t, \by) = y^0\big\},\quad  \overline{\Oc}(t)\coloneqq \big\{\by\in H: y^0 = \ol w(t,\by)\big\}.
\]
Within the model of \cite{hernandez2024closed}, which relies on the previous work of \citeauthor*{bouchard2010optimal} \cite{bouchard2010optimal}, one is then naturally led to the formal dynamic programming equation
\begin{align}\label{PrincipalPDE}
\begin{cases}
\displaystyle- \pa_t V(t,x,\by) + F\big(t, x, \by, \big(\pa_x V, \pa_{xx}^2 V, D_\by V, D_{\by\by}^2 V, \pa_x D_\by V\big)(t,x,\by)\big) = 0,\; (t,x)\in[0,T)\times\R,\; \by\in\Oc(t), \\[0.3em]
 \displaystyle- \pa_t V(t,x,\by) + \ul F\big(t, x, \by, \big(\pa_x V, \pa_{xx}^2 V, D_\by V, D_{\by\by}^2 V, \pa_x D_\by V\big)(t,x,\by)\big) = 0, \; (t,x)\in[0,T)\times\R,\; \by\in\underline{\Oc}(t),  \\[0.3em]
\displaystyle  - \pa_t V(t,x,\by) + \ol F\big(t, x, \by, \big(\pa_x V, \pa_{xx}^2 V, D_\by V, D_{\by\by}^2 V, \pa_x D_\by V\big)(t,x,\by)\big) = 0, \; (t,x)\in[0,T)\times\R,\; \by\in\overline{\Oc}(t), \\[0.3em]
\displaystyle  V(T, x, \by) = U_\smallertext{\rm P}(x, y^0),
\end{cases}
\end{align}
where
\begin{gather*}
F\big(t, x, \by, p, A, \bp, \mathbf{A}, \mathbf{q}\big) \coloneqq \sup_{\bz \in H}\bigg\{ a_t^{\star}(z^t)p + \langle c_t^{\star}(\cdot, z^t), \bp \rangle_H + \frac{1}{2}\big(A + \langle \bz, \mathbf{A} \bz \rangle_H + 2 \langle \bz, \mathbf{q} \rangle_H \big)   \bigg\}, \\[0.5em]
\ul F\big(t, x, \by, p, A, \bp, \mathbf{A}, \mathbf{q}\big) \coloneqq \sup_{\bz \in \ul H(t, \by, \ul w)}\bigg\{ a_t^{\star}(z^t)p + \langle c_t^{\star}(\cdot, z^t), \bp \rangle_H + \frac{1}{2}\big(A + \langle \bz, \mathbf{A} \bz \rangle_H + 2 \langle \bz, \mathbf{q} \rangle_H \big)   \bigg\}, \\[0.5em]
\ol F\big(t, x, \by, p, A, \bp, \mathbf{A}, \mathbf{q}\big) \coloneqq \sup_{\bz \in \ol H(t, \by, \ol w)}\bigg\{ a_t^{\star}(z^t)p + \langle c_t^{\star}(\cdot, z^t), \bp \rangle_H + \frac{1}{2}\big(A + \langle \bz, \mathbf{A} \bz \rangle_H + 2 \langle \bz, \mathbf{q} \rangle_H \big)   \bigg\},
\end{gather*}
with the sets $\ul H$ and $\ol H$ defined by
\begin{gather*}
\ul H(t, \by,  w)  \coloneqq \bigg\{\bz \in H : \langle \bz, D_\by w(t, \by) \rangle_H - z^0 = 0, \; \mbox{and} \; c_t^{\star}(0, z^t) - \langle c_t^{\star}(\cdot, z^t), D_\by w(t, \by) \rangle_H + \frac{1}{2}\langle \bz, D_{\by\by}^2 w(t, \by) \bz \rangle_H \ge 0 \bigg\}, \\
\ol H(t, \by,  w)  \coloneqq \bigg\{\bz \in H : \langle \bz, D_\by w(t, \by) \rangle_H - z^0 = 0, \; \mbox{and} \; c_t^{\star}(0, z^t) - \langle c_t^{\star}(\cdot, z^t), D_\by w(t, \by) \rangle_H + \frac{1}{2}\langle \bz, D_{\by\by}^2 w(t, \by) \bz \rangle_H \le 0 \bigg\},
\end{gather*}
for all $(t, \by) \in [0,T] \times H$ and smooth $w : [0,T] \times H \to \dbR$.

\begin{remark}
The discussion above is heuristic at three distinct levels.

\medskip
$(i)$ We do not prove the identification of the principal's reachability set with the band
\[
\big\{\by\in H:\ul w(t,\by)\le y^0\le \ol w(t,\by)\big\}.
\]

\medskip
$(ii)$ Even assuming this identification, we do not establish the dynamic programming principle for the resulting state-constrained control problem on $[0,T]\times\dbR\times H$. 

\medskip
$(iii)$ We do not prove comparison, regularity, or verification results for the corresponding boundary-value problem. These would be required to turn the formal {\rm PDE} derivation into a full theorem for the principal's problem.

\medskip
Accordingly, the {\rm PDE} computations of this subsection should be read as formal motivation only. In {\rm\Cref{sec:utilexam}} below, we only use the resulting heuristic structure and do not need a full verification theorem.
\end{remark}

\paragraph*{General dynamic programming equation.}
In the general case, the reachability set $\cV$ might not be represented through functions of the state process. We therefore provide a general equation with geometric constraints. The intuition is the same as above: if the state variable lies in the interior of the reachability set, then we obtain the standard dynamic programming equation. At the boundary, one does not need to kill the whole volatility: tangential volatility is admissible. What matters is the usual viability condition, namely that the normal component of the volatility vanishes and that the second-order contribution does not point outside the reachable set. 

\medskip
Let $w : [0,T] \times H \longrightarrow \dbR$ be such that $\cV(t) = \{\by \in H : w(t, \by) \le 0\}$ (take for example $w(t,\by) = 1 - \1_{\by \in \cV(t)}$). We formally derive the following system of HJB equations
\begin{align}\label{PrincipalGeometricPDE}
\begin{cases}
\displaystyle- \pa_t V(t,x,\by) + F\big(t, x, \by, \big(\pa_x V, \pa_{xx}^2 V, D_\by V, D_{\by\by}^2 V, \pa_x D_\by V\big)(t,x,\by)\big) = 0,\; (t,x)\in[0,T)\times\R,\; \by\in \mathrm{Int}(\cV(t)), \\[0.3em]
 \displaystyle- \pa_t V(t,x,\by) + \pa F\big(t, x, \by, \big(\pa_x V, \pa_{xx}^2 V, D_\by V, D_{\by\by}^2 V, \pa_x D_\by V\big)(t,x,\by)\big) = 0, \; (t,x)\in[0,T)\times\R,\; \by \in \pa \cV(t),  \\[0.3em]
\displaystyle  V(T, x, \by) = U_\smallertext{\rm P}(x, y^0),
\end{cases}
\end{align}
where $F$ is defined as above and
\begin{gather*}
\pa F\big(t, x, \by, p, A, \bp, \mathbf{A}, \mathbf{q}\big) \coloneqq \sup_{\bz \in \pa H(t, \by, w)}\bigg\{ a_t^{\star}(z^t)p + \langle c_t^{\star}(\cdot, z^t), \bp \rangle_H + \frac{1}{2}\big(A + \langle \bz, \mathbf{A} \bz \rangle_H + 2 \langle \bz, \mathbf{q} \rangle_H \big)   \bigg\},
\end{gather*}
with
\begin{gather*}
\pa H(t, \by,  w)  \coloneqq \bigg\{\bz \in H : \langle \bz, D_\by w(t, \by) \rangle_H = 0, \; \mbox{and} \;   \langle c_t^{\star}(\cdot, z^t), D_\by w(t, \by) \rangle_H + \frac{1}{2}\langle \bz, D_{\by\by}^2 w(t, \by) \bz \rangle_H \le 0 \bigg\},
\end{gather*}
for all $(t, \by) \in [0,T] \times H$ and smooth $w : [0,T] \times H \longrightarrow \dbR$.

\subsubsection{The exponential utility case}\label{sec:utilexam}

In this section, we assume the agent's utility function is given by
\[
V_t^A(\xi) \coloneqq  \sup_{\a \in \cA} \dbE^{\P^\smalltext{\alpha}}\bigg[ U_\smallertext{\rm A}\bigg(f(T-t)\xi - \frac 1 2 \int_t^T f(r-t)\a_r^2 \mathrm{d}r \bigg) \bigg], 
\]
with $U_{\smallertext{\rm A}}(x) \coloneqq  -\frac{1}{\g_\smallertext{\rm A}}\mathrm{e}^{- \g_\smalltext{\rm A} x}$ and $f(t) \coloneqq  \mathrm{e}^{-\rho t}$, with positive $\g_\smallertext{\rm A}$ and  $\rho$. We also assume that the principal is risk neutral, that is
\[
 U_\smallertext{\rm P}(x, y) \coloneqq  x-y, \; \mbox{for all} \; (x,y) \in \dbR^2. 
 \]
We shall discuss the conjecture made in \cite{hernandez2024time}, claiming that the optimal incentive $\bZ^\star$ designed by the principal in this setting is deterministic, similarly to the case where $U_\smallertext{A} = \mathrm{Id}$. In this setting, the dynamics controlled by the principal writes
\[
\mathrm{d}Y_t^s = \bigg(\mathrm{e}^{-\rho(t-s)} \frac{(Z_t^t)^2}{2} + \gamma \frac{(Z_t^s)^2}{2}\bigg)\mathrm{d}t + \sigma Z_t^s \mathrm{d}W_t.
\]
Assume there exists a deterministic optimal control $\bZ^\star_t = \bz_t$. We first deduce from the constraint
\[
 Y_T^{s, \bZ^\star} = \frac{f(T-s)}{f(T)}Y_T^{0, \bZ^\star} = \mathrm{e}^{\rho s}Y_T^{0, \bZ^\star} ,\; \mbox{for all $s \in [0,T]$},
 \]
that $\int_0^T (\bz_r^s - \mathrm{e}^{\rho s} \bz_r^0)\mathrm{d}W_s$ is deterministic, from which we deduce that $\bz_t^s = \mathrm{e}^{\rho s}\bz_t^\mathrm{0}$ for Lebesgue--a.e. $(t ,s) \in [0,T]^2$.

\medskip
 Then, given the dynamics of $(X, \bY)$, it is clear that the dynamic value function $V : [0,T] \times \dbR \times H \longrightarrow \dbR$ of the principal is smooth, as
\[
 V(t, x, \by) \coloneqq  \dbE^\P\big[ X_T^{t, x, \bZ^{\star}} - Y_T^{0, t, \by, \bZ^{\star}} \big] = x - y^0 + \int_t^T \bigg(z_r^{r} - \mathrm{e}^{-\rho r} \frac{(z_r^r)^2}{2} - \gamma_{\smallertext{\rm A}} \frac{(z_r^0)^2}{2} \bigg)\mathrm{d}r. 
 \]
Therefore, by the dynamic programming equation \eqref{PrincipalPDE}, if $(x, \by) \in \mathrm{Int}(\cV(t))$, we have
\[
 -\pa_t V(t,x, \by) - \sup_{\bZ \in H} \bigg\{ \pa_x V(t, x, \by) Z^t + \bigg\langle D_\by V(t, x, \by), \mathrm{e}^{-\rho (t-\cdot)} \frac{(Z^t)^2}{2} + \g \frac{\bZ^2}{2} \bigg\rangle_H \bigg\} = 0 ,
\]
 observing that the second order derivatives of $V$ are equal to $0$. Given that $\pa_x V = 1$ and $D_\by V = \bv_0$, the first order condition in the Hamiltonian writes
\[\bv_t - \mathrm{e}^{- \rho t} z^t \bv_t - \gamma z^0 \bv_0 =  \bv_t - z^0 \bv_t - \gamma z^0 \bv_0 = 0,
\] which is clearly impossible. Therefore, the conjecture formulated in \cite{hernandez2024time} is false whenever there exists $t \in [0,T)$ such that $\mathrm{Int}(\cV(t)) \neq \emptyset$. 

\begin{remark}\label{rem:smoothspuersol}
Of course, the previous argument does not invalidate the conjecture if it turns out that $\mathrm{Int}(\cV(t))$ is empty, which we have not been able to rule out. One possible strategy to prove that this cannot happen would be to find a smooth super-solution $w^\smallertext{+}$ of the {\rm PDE} associated to the reachability set. Indeed, in this case the non-empty domain $w^\smallertext{+}\leq 0$ would be included in the reachability set. Despite several attempts, we have not been able to construct such a function.
\end{remark}

\subsubsection{Study for a special discount factor}

We end the discussion on the principal--agent problem with an example where a reduction of dimension can be obtained. Consider the case where the agent's time-inconsistency only comes from the presence of a non-exponential discount factor
\[
 U_A(s, \xi) = f(T-s)U_{\smallertext{\rm A}}(\xi), \; c_t(s, a) = f(t-s)c_t(a). 
 \]
Moreover, we assume that the function $f$ takes the form
\[ f(t) = \sum_{k=1}^N \b_k \mathrm{e}^{- \rho_\smalltext{k} t}, 
\]
for some $N \in \dbN^{\star}$, where the sequences $(\b_k)_{k\in\{1,\dots,N\}}$ and $(\rho_k)_{k\in\{1,\dots,N\}}$ are non-negative, such that $\sum_{k=1}^N \b_k = 1$ and $\rho_k \neq \rho_j$ for all $(k,j)\in\{1,\dots,N\}^2$ with $k\neq j$. We then have the following inclusion.
\begin{lemma}\label{lem:finite-dimension}
For all $k \in \{1, \dots, N\}$, denote $\phi_k(s) \coloneqq \b_k \mathrm{e}^{\rho_\smalltext{k} s}$. Then we have, for all $t \in [0,T]$
\[
 \big(\bY_t^\bZ, \bZ_t\big) \in \operatorname{span}\{ \phi_1, \dots, \phi_N \}\times \operatorname{span}\{ \phi_1, \dots, \phi_N \}. 
 \]
\end{lemma}
\begin{proof}
First observe that the stochastic target constraint \eqref{stochastic-target1} writes in this context
\[Y_T^{s, \bZ} = \frac{f(T-s)}{f(T)} Y_T^{0, \bZ} = \frac{1}{f(T)}\sum_{k=1}^N \phi_k(s)\mathrm{e}^{-\rho_k T} Y_T^{0, \bZ}, \]
for all $s \in [0,T]$. Taking the conditional expectation with respect to $\cF_t$ in the above equality, we obtain
\[
Y_t^{s, \bZ} + \dbE^\P\bigg[ \int_t^T f(T-s)c_r^{\star}(Z_r^r)\mathrm{d}r \bigg| \cF_t \bigg] =  \frac{1}{f(T)}\sum_{k=1}^N \phi_k(s)\mathrm{e}^{-\rho_\smalltext{k} T} \Bigg(Y_t^{0,\bZ} + \dbE^\P\bigg[ \int_t^T f(T)c_r^{\star}(Z_r^r)\mathrm{d}r \bigg| \cF_t \bigg]\Bigg), 
\]
which means that $\bY_t^\bZ \in \operatorname{span}\{ \phi_1, \dots, \phi_N \}$ since $f(T-\cdot) \in \operatorname{span}\{ \phi_1, \dots, \phi_N \}$. Furthermore, as both $\bY^\bZ$ and its drift term lie in this space, we have
\[
 \int_0^t \bZ_r \mathrm{d}W_r \in \operatorname{span}\{ \phi_1, \dots, \phi_N \}, \; \mbox{for all $t \in [0,T]$}, 
 \]
which in turn implies that $\bZ_t \in \operatorname{span}\{ \phi_1, \dots, \phi_N \}$ for all $t$.
 \end{proof}


The main consequence of \Cref{lem:finite-dimension} is that the principal's problem becomes finite-dimensional. Indeed, admissible controls write
\[ Z_\cdot^s = \sum_{k=1}^N \phi_k(s)\tilde Z_\cdot^k, \]
where the processes $(\tilde Z_\cdot^k)_{k\in\{1,\dots,N\}}$ are $\dbR$-valued and adapted. Similarly, we have $\bY = \sum_{k=1}^N \phi_k \tilde Y^k$,
where
\[
\tilde Y_t^k \coloneqq \tilde Y_0^k + \int_0^t \mathrm{e}^{-\rho_\smalltext{k} r}c_r^\star\Bigg(\sum_{\ell=1}^N \phi_\ell(r) \tilde Z_r^\ell\Bigg)\mathrm{d}r + \int_0^t \tilde Z_r^k \mathrm{d}W_r,
\]
for all $k \in \{1, \dots, N\}$, observing that $Z_r^r = \sum_{k=1}^N \phi_k(r) \tilde Z_r^k$. Then, decomposing $f(T-\cdot)$ in the basis $(\phi_1, \dots, \phi_N)$, we see that the stochastic target constraint \eqref{stochastic-target1} reformulates as
\begin{equation}\label{stochastic-target3}
\tilde Y_T^k = \b_k\frac{\mathrm{e}^{-\rho_\smalltext{k} T}}{f(T)}\sum_{\ell=1}^N \b_\ell \tilde Y_T^\ell,\; \mbox{for all $k \in \{1,\dots,N\}$}, 
\end{equation}
as $\bY_T^0 = \sum_{k=1}^N \b_k \phi_k(0) \tilde Y_T^k = \sum_{k=1}^N \b_k \tilde Y_T^k$. We easily see that this constraint means that the vector $\tilde Y_T \coloneqq  (\tilde Y_T^1, \dots, \tilde Y_T^N)$ must belong to the line $\cD$ in $\dbR^N$ defined by the system of equations
\[
 y_1 = \frac{\b_1}{\b_k}\mathrm{e}^{(\rho_\smalltext{k}-\rho_\smalltext{1})T}y_k,\; \mbox{for all $k \in \{2, \dots, N\}$}, 
 \]
or, equivalently, that
\[
 \max_{k \in \{2, \dots, N\}} \bigg\{\frac{\b_1}{\b_k}\mathrm{e}^{(\rho_\smalltext{k}-\rho_\smalltext{1})T}y_k\bigg\} \le y_1 \le \min_{k \in \{2, \dots, N\}} \bigg\{\frac{\b_1}{\b_k}\mathrm{e}^{(\rho_\smalltext{k}-\rho_\smalltext{1})T}y_k\bigg\}, 
 \]
which writes again as the combination of two epigraph-type constraints. Then, the Principal must solve the finite dimensional control problem under stochastic target constraint
\[
\sup_{\tilde Z : \tilde Y_T \in \cD} \dbE^\P\bigg[U_{\smallertext{\rm P}}\bigg(X_T, \sum_{k=1}^N \b_k \tilde Y_T^k\bigg)\bigg].
\]
This finite-dimensional reduction should be viewed as the identification of an explicit finite-dimensional invariant subspace of the lifted Hilbert state. It therefore provides a tractable special case of the present framework rather than a disconnected problem. Note that it falls under the setting of \citeauthor*{hernandez2024closed} \cite{hernandez2024closed}, who derive the corresponding dynamic programming equation in finite dimension.

\section{Markovian representation and approximation}\label{sec:markovian-representation}
In this section, we discuss how our framework provides a natural Markovian approximation for the Volterra-type dynamics \eqref{VolterraSDE}. When $b$ and $\si$ are as in \eqref{VolterraSDE}, we recall that the Volterra-type process $X$ is related to the infinite dimensional process $\bX$ from \eqref{liftedSDE1} in the following way
\[
X_t = \phi(t, \bX_t),\; \mbox{for all $t \in [0,T]$}, 
\]
where the mapping $\phi : [0,T] \times H \longrightarrow \dbR$ is defined by 
\[
 \phi(t, \bx) \coloneqq \tilde x^t, \; \mbox{for all $(t, \bx) \in [0,T] \times H$.} 
\]
By \eqref{ineq:inf-sobolev}, the mapping $\phi(t, \cdot)$ is a continuous linear form on $H$ for all $t \in [0,T]$. Therefore, by Riesz's representation theorem, there exists an $H$-valued mapping $t \longmapsto  v_t$ such that 
\[
\phi(t, \bx) = \langle v_t, \bx \rangle_H,\; \mbox{for all $(t, \bx) \in [0,T] \times H$}.
\]
Note that $v_t$ is a Sobolev solution of the equation
\begin{equation}\label{pde-v}
v_t - \pa_s^2 v_t = \d_t ,\; \mbox{for all $t \in [0,T]$}.
\end{equation}
Note that $t$ is only a parameter here. Solving the equation on $(0,t)$ and $(t,T)$, and using the continuity condition of $s \longmapsto  v_t^s$ in $t$, we see that $v$ has the form
\begin{equation}\label{function-v}
 v_t^s = \frac{1}{2}\big(A\mathrm{e}^{s} + B\mathrm{e}^{-s}\big) + \frac{1}{2}\big(\mathrm{e}^{s-t} -\mathrm{e}^{-(s-t)}\big)\1_{\{s \ge t\}}, \; \mbox{$(A, B) \in \dbR^2$}. 
 \end{equation}

Let $(e_k)_{k \in \mathbb N^\smalltext{\star}}$ be an orthonormal basis of $H$. Denote for any $t\in[0,T]$ and any $k\in\mathbb N^\star$ by $v_t^k$ and $X_t^k$ the projections of $v_t$ and $\bX_t$ on $e_k$. Then
\begin{equation}\label{X-scalar} 
X_t = \sum_{k=1}^\infty v_t^k X_t^k.
\end{equation}
For all $k \in\mathbb N^\star$, $X^k$ solves the SDE
\[
X_t^k = x_k + \int_0^t b_r^k(X_r^0, \a_r)\mathrm{d}r + \int_0^t \si_r^k(X_r^0, \a_r)\mathrm{d}W_r,
\]
where
\[
x_k \coloneqq  \langle \bx, e_k \rangle_H, \;
b_t^k(x,a) \coloneqq   \langle b_t(\cdot, x,a), e_k \rangle_H, \;
\si_t^k(x,a) \coloneqq   \langle \si_t(\cdot, x,a), e_k \rangle_H.
\]
Moreover, for every $k\in\mathbb N^\star$, the map $t\longmapsto v_t^k$ is of class $C^1$. Therefore, differentiating \eqref{X-scalar} formally yields
\[
 \mathrm{d}X_t^0 = \bigg(\sum_{k=1}^\infty \pa_t v_t^k X_t^k + \sum_{k=1}^\infty v_t^k b_t^k(X_t^0,\a_t)\bigg)\mathrm{d}t + \sum_{k=1}^\infty v_t^k \si_t^k(X_t^0, \a_t)\mathrm{d}W_t,
\]
where $X^0\coloneqq X$. We are thus led to the infinite-dimensional Markovian system
\begin{equation}\label{Markov}
\begin{cases}
\displaystyle X_t^0 = x_0 + \int_0^t \bigg(\sum_{k=1}^\infty \pa_t v_r^k X_r^k + \sum_{k=1}^\infty v_r^k b_r^k(X_r^0, \a_r)\bigg)\mathrm{d}r + \int_0^t \sum_{k=1}^\infty v_r^k \si_r^k(X_r^0, \a_r)\mathrm{d}W_r, \\[0.8em]
\displaystyle X_t^k = x_k + \int_0^t b_r^k(X_r^0, \a_r)\mathrm{d}r + \int_0^t \si_r^k(X_r^0, \a_r)\mathrm{d}W_r,\; k \in\mathbb N^\star.
\end{cases}
\end{equation}
A natural finite-dimensional approximation is obtained by truncating the sums in \eqref{Markov}. Thus, for $n \in\mathbb N^\star$, we consider the $(n+1)$-dimensional dynamics
\begin{equation}\label{Markov-approximate}
\begin{cases}
\displaystyle X_t^{0,n} = x_0 + \int_0^t \bigg(\sum_{k=1}^n \pa_t v_r^k X_r^{k,n} + \sum_{k=1}^n v_r^k b_r^k(X_r^{0,n}, \a_r)\bigg)\mathrm{d}r + \int_0^t \sum_{k=1}^n v_r^k \si_r^k(X_r^{0,n}, \a_r)\mathrm{d}W_r, \\[0.8em]
\displaystyle X_t^{k,n} = x_k + \int_0^t b_r^k(X_r^{0,n}, \a_r)\mathrm{d}r + \int_0^t \si_r^k(X_r^{0,n}, \a_r)\mathrm{d}W_r,\; k\in\{1,\dots,n\}.
\end{cases}
\end{equation}

\begin{proposition}\label{prop:markov-approximation}
We have
\[
 \lim_{n \to \infty} \dbE^\P\bigg[ \sup_{t \in [0,T]} \big| X_t^{0,n} - \bar X_t^n \big|^2 \bigg] = 0,
 \;
 \lim_{n \to \infty} \dbE^\P\big[ | X_t^{0,n} - X_t^0 |^2 \big] = 0,\; \forall t \in [0,T],
 \]
where $\bar X_t^n \coloneqq \sum_{k=1}^n v_t^k X_t^k$.
\end{proposition}

\begin{proof}
Introduce, for $n\in\mathbb N^\star$,
\[
\bar X_t^n \coloneqq  \sum_{k=1}^n v_t^k X_t^k,
\;
R_t^n \coloneqq  \sum_{k=n+1}^\infty v_t^k X_t^k.
\]
Then $X_t^0=\bar X_t^n+R_t^n$, and therefore
\begin{equation}\label{est1}
| X_t^{0,n} - X_t^0 |^2 \le 2 | X_t^{0,n} - \bar X_t^n |^2 + 2| R_t^n |^2.
\end{equation}

By \eqref{Markov} and \eqref{Markov-approximate},
\[
X_t^{0,n} - \bar X_t^n = \int_0^t \sum_{k=1}^n v_s^k \big(b_s^k(X_s^{0,n}, \a_s) - b_s^k(X_s^0, \a_s)\big)\mathrm{d}s + \int_0^t \sum_{k=1}^n v_s^k\big(\si_s^k(X_s^{0,n}, \a_s) - \si_s^k(X_s^0, \a_s)\big)\mathrm{d}W_s.
\]
Hence, by Cauchy--Schwarz and Burkholder--Davis--Gundy
\begin{align*}
\dbE^\P\bigg[ \sup_{u \in [0,t]} \big| X_u^{0,n} - \bar X_u^{n} \big|^2\bigg]
&\le C \int_0^t \sum_{k=1}^n |v_s^k|^2 \dbE^\P\Bigg[\sum_{k=1}^n \big| b_s^k(X_s^{0,n}) - b_s^k(X_s^0) \big|^2 \\
&\quad+ \sum_{k=1}^n \big| \si_s^k(X_s^{0,n}, \a_s) - \si_s^k(X_s^0, \a_s) \big|^2\Bigg]\mathrm{d}s \\
&\le C \int_0^t \|v_s\|_H^2 \dbE^\P\Big[\|b_s(\cdot,X_s^{0,n}, \a_s)-b_s(\cdot,X_s^0, \a_s)\|_H^2 \\
&\quad+ \|\si_s(\cdot,X_s^{0,n}, \a_s)-\si_s(\cdot,X_s^0, \a_s)\|_H^2\Big]\mathrm{d}s.
\end{align*}
Since $b$ and $\si$ are Lipschitz-continuous as $H$-valued mappings and $t\longmapsto v_t$ is continuous on $[0,T]$, we deduce from \eqref{est1} that
\begin{equation}\label{est2}
\dbE^\P\bigg[ \sup_{u \in [0,t]} \big| X_u^{0,n} - \bar X_u^{n} \big|^2\bigg]
\le C \int_0^t \Big( \dbE^\P\big[| X_s^{0,n} - \bar X_s^n |^2\big] + \dbE^\P\big[| R_s^n |^2\big]\Big)\mathrm{d}s.
\end{equation}
Gronwall's lemma then yields
\begin{equation}\label{est3}
\dbE^\P\bigg[ \sup_{u \in [0,T]} \big| X_u^{0,n} - \bar X_u^{n} \big|^2\bigg]
\le C \int_0^T \dbE^\P\big[ | R_s^n |^2 \big]\mathrm{d}s.
\end{equation}

Finally
\[
 |R_s^n|^2 \le \sum_{k=n+1}^\infty |v_s^k|^2\sum_{k=n+1}^\infty |X_s^k|^2 \le \|v_s\|_H^2\|\bX_s\|_H^2,
\]
and $R_s^n\longrightarrow0$ for $\mathrm{d}s\otimes \dbP$--a.e. $(s,\omega)$. By dominated convergence,
\[
 \int_0^T \dbE^\P\big[ | R_s^n |^2 \big]\mathrm{d}s \longrightarrow 0.
\]
Combining this with \eqref{est3} proves the first convergence. The second one then follows from \eqref{est1}.
\end{proof}

{ \paragraph*{Application to control problems.} Consider the controlled problem defined by:
\[
 V_0^n \coloneqq \sup_{\a \in \cA} \dbE^\mathbb P\big[g(X_T^0)  \big].  
 \]
Clearly, this is a Markovian control problem with state process $(X^{0,n}, \dots, X^{n,n})$, and we have $V_0^n = V^n(0, x_0, \dots, x_n)$, where $V^n$ is solution of the finite-dimensional HJB equation
\begin{align*} 
&- \pa_t V^n - \sup_{a \in \dbA}\Bigg\{\bigg(\sum_{k=1}^n \pa_t v_t^k x_k + v_t^k b_t^k(x_0, a) \bigg)\pa_{x_0} V^n + \sum_{k=1}^n b_t^k(x_0, a) \pa_{x_k}V^n \\
&\quad+ \frac{1}{2} \sum_{i=1}^n \sum_{j=1}^n \si_t^i(x_0, a) \si_t^j(x_0, a) \pa_{x_i x_j}^2 V^n    \Bigg\} = 0, \; V^n(T, \cdot) = g,
\end{align*}
where we set $\si_t^0(x_0, a) \coloneqq \sum_{i=1}^n v_t^i \si_t^i(x_0,a)$. 

\medskip
Assuming that $g$ is Lipschitz-continuous, it is clear by \Cref{prop:markov-approximation} that $V_0^n \underset{n \to \infty}{\longrightarrow} V_0$, where $V_0$ is given by \eqref{control-pb}.  $V_0^n$ can be approximated by standard numerical methods for finite-dimensional Markovian control problems. However, because $n$ is typically large, it is generally preferable to rely on numerical methods tailored to high-dimensional settings, such as recent algorithms combining dynamic programming and neural networks, see \emph{e.g.} \citeauthor*{hure2021deep} \cite{hure2021deep}. 
}

\begin{remark}
The convergence rate of our Markovian approximation typically depends on the remainder
\[
R_t^n = \sum_{k=n+1}^\infty v_t^k X_t^k,
\]
which in turn depends strongly on the chosen Hilbert basis $(e_k)_{k\in\mathbb N^\smalltext{\star}}$. Therefore, identifying the best convergence rate for our approximation amounts to identifying a basis adapted to the geometry of the lifted state. This question is far beyond the scope of the present paper. We only note that several previous works have studied multidimensional Markovian approximations of Volterra-type dynamics $($typically for monotone kernels$)$, see for instance {\rm \citeauthor*{abi2019multifactor} \cite{abi2019multifactor}, \citeauthor*{abi2021linear} \cite{abi2021linear}, \citeauthor*{harms2021strong} \cite{harms2021strong}, \citeauthor*{alfonsi2024approximation} \cite{alfonsi2024approximation}, and \citeauthor*{bayer2023markovian} \cite{bayer2023markovian}}; see also \citeauthor*{khabou2025markov} \cite{khabou2025markov} in the context of Hawkes processes. 
\end{remark}

\section{Comparison with other results in the literature}\label{sec:comparison}

\subsection{Lifting approach}

In this section, we discuss how our methodology completes pre-existing results regarding the control of stochastic Volterra integral equations. It is common in the literature (see e.g. \citeauthor*{abi2021linear} \cite{abi2021linear} or \citeauthor*{hamaguchi2023markovian} \cite{hamaguchi2023markovian}) to assume that the control process has the following dynamics
\begin{equation}\label{SDE-kernel}
X_t^\a = x + \int_0^t K(t-r)\big(b(X_r^\a)\mathrm{d}r + \si(X_r^\a)\mathrm{d}W_r\big),
\end{equation}
where the kernel $K$ writes $K(t) \coloneqq \int_{\dbR} \mathrm{e}^{-\th t} \m(d\th)$ for some signed measure $\m$. Note that such kernels may be singular at $0$. In \cite{di2023lifting}, \citeauthor{di2023lifting} generalise this structure by writing 
\begin{equation}\label{kernel-structure}
K(t) = \langle g, \cS_t \n \rangle_{Y \times Y^\smalltext{\star}},
\end{equation}
where $Y$ is a UMD Banach space, $Y^\star$ its dual, $\cS$ a semi-group acting on $Y^\star$ and $g$, $\n$ two elements in $Y$ and $Y^\star$ respectively. We easily see that the examples of \cite{abi2021linear} and \cite{hamaguchi2023markovian} are covered by this structure. Then, switching the integrals in $\mathrm{d}r$ and $\mathrm{d}W_r$ and the duality bracket $\langle \cdot, \cdot \rangle_{Y \times Y^\smalltext{\star}}$, we can write
\[
X_t^\a = x + \langle g, \bX_t^\a \rangle_{Y \times Y^\smalltext{\star}}, 
\]
where $\bX^\a$ is a $Y^\star$-valued process satisfying the infinite dimensional Markovian SDE
\[
 \bX_t^\a = \int_0^t \cA^\star \bX_r \mathrm{d}r + \int_0^t \n b\big(\langle g, \bX_r^\a \rangle_{Y \times Y^\smalltext{\star}}\big)\mathrm{d}r + \int_0^t \n \si\big(\langle g, \bX_r^\a \rangle_{Y \times Y^{\smalltext{\star}}}\big)\mathrm{d}W_r, 
 \]
 thanks to the semi-group structure of $\cS$. The problem of controlling $X^\a$ is then reduced to the problem of controlling the infinite dimensional Markovian dynamics $\bX^\a$. 

\medskip
Our approach generalises this reduction for the case of regular kernels. In particular, we do not require any semi-group structure. Recall our general Volterra-type dynamics
\[
X_t^\a = x + \int_0^t b_r(t, X_r^\a)\mathrm{d}r + \si_r(t, X_r^\a)\mathrm{d}W_r. 
\]
Then, assuming that $b$ and $\si$ have Sobolev regularity in the `Volterra time' $t$, we may write
\[
b_r(t, x) = \langle b_r,(\cdot, x), v_t \rangle_H, \; \mbox{and} \; \si_r(t, x) = \langle \si_r,(\cdot, x), v_t \rangle_H,  
\]
where $v_t$ is defined by \eqref{function-v} for all $t \in [0,T]$. Then, observing that $x = \langle \bp(x), v_t \rangle_H$, we obtain
\[
 X_t^\a = \langle \bX_t^\a, v_t \rangle_H, 
\]
where $\bX^\a$ follows the $H$-valued Markovian SDE:
\[
 \bX_t^\a = \bp(x) + \int_0^t b_r\big(\cdot, \langle \bX_r^\a, v_r \rangle_H\big)\mathrm{d}r + \int_0^t \si_r\big(\cdot, \langle \bX_r^\a, v_r \rangle_H \big)\mathrm{d}W_r. 
 \]

\subsection{PDE approach} 

We now mention other works connecting Volterra dynamics to partial differential equations, which are often used jointly to a lifting approach. We first mention the contribution of \citeauthor*{viens2019martingale} \cite{viens2019martingale}, whose methodology is in spirit the closest to ours. Their idea is the following: given a process $X$ of the form
\begin{equation}\label{X-ViensZhang}
X_t = \int_0^t K(r,t) \mathrm{d}W_r, 
\end{equation}
with $K$ a possibly singular kernel, one wants to find a PDE characterising the process
\begin{equation}\label{Y-ViensZhang}
Y_t \coloneqq  \dbE^\P\bigg[ \int_t^T f(r, X_r)\mathrm{d}r + g(X_T) \bigg| \cF_t\bigg]. 
\end{equation}
As usual, the trick is to find a function $u : [0,T] \times \dbR \longrightarrow \dbR$ and an adapted process $\tilde X$ so that $Y_t = u\big(t, \tilde X_t\big)$. Of course, the connection between the stochastic representation above and the PDE is derived through an appropriate form of Itô's formula, for which a semi-martingale structure on $\tilde X$ is necessary. Since $X$ defined in \eqref{X-ViensZhang} is obviously not a semi-martingale, the authors of \cite{viens2019martingale} introduce the family of auxiliary processes $\Th$ defined by
\begin{equation}\label{lifting-ViensZhang}
 \Th_s^t \coloneqq \int_0^t K(r,s)\mathrm{d}W_r,\; \mbox{for all $s \ge t$},
 \end{equation}
and prove that 
\[
Y_t = u(t, X \otimes_t \Th^t), 
\]
where $(X \otimes_t \Th^t)_s \coloneqq X_s \1_{\{s < t\}} + \Th_s^t \1_{\{s \ge t\}}$ and $u$ satisfies the path-dependent PDE
\begin{equation}\label{PPDE-ViensZhang}
\pa_t u(t, \o) + \frac 1 2 \pa_{\o\o}^2 u(t, \o)(K(t, \cdot), K(t,\cdot)) + f(t,\o) = 0, \ u(t, \o) = g(T, \o),
\end{equation}
for all $(t,\o) \in [0,T] \times C^0([0,T], \dbR)$, where the path derivative is defined in the spirit of Dupire \cite{dupire2009functional}, on the space of càdlàg paths. This characterisation assumes that the derivatives of the function $u$ are well defined. This approach has been extended by \citeauthor*{wang2022path} in \cite{wang2022path} to the case of Volterra-type forward--backward SDEs. In \cite{wang2023linear}, \citeauthor*{wang2023linear} uses this dynamic programming equation to characterise the value function of a linear--quadratic problem by a system of path-dependent Riccati equations. 

\medskip
Let us now compare this approach with our work. We easily see that the family of random variables $(\Th_s^t)_{s \in [t,T]}$ corresponds to our $(X_t^s)_{s \in [t,T] }$, the only difference being that the domain of $(t,s)$ in \cite{viens2019martingale} is triangular (\emph{i.e.}, one requires $s > t$) due to the potential singularity of $K$, whereas it is rectangular in our framework. Furthermore, for fixed $t$, we easily see that $u(t, \o) = u(t, \o \1_{[t,T]})$ in this example. Therefore, introducing $\Th$ defined by \eqref{lifting-ViensZhang} is almost equivalent to our lifting from $X$ to $\bX = (X^{s})_{s \in [0,T]}$, with the difference that the space where the new state process takes its values is not the same. In \cite{viens2019martingale}, this would be the Banach space of càdlàg paths, whereas we chose the Hilbert space of Sobolev functions on $[0,T]$. Although this forces us to have stronger regularity assumptions on the coefficients of our dynamics, this dramatically reduces the need for regularity on $u$, as we can resort to the standard theory of viscosity solution on Hilbert space to derive our dynamics programming equation. 

\medskip
We mention again the recent contribution of \citeauthor*{di2023lifting} \cite{di2023lifting}, who characterise the solution of a Volterra control problem by means of a backward SDE and as the mild solution to the corresponding dynamic programming equation. More precisely, as highlighted in Section \ref{sect:BSDE}, they consider the controlled dynamics \eqref{Volterra-BSDE}, with the extra assumption that the kernel $K$ has the structure \eqref{kernel-structure}. They are then able to prove that, given that the value function has a first order Gâteaux derivative, the problem is characterised by mean of a semi-linear PDE on an UMD Banach space, satisfied in the sense of mild solutions. In our context, since we operate our lifting in a Hilbert space, this connection between backward SDEs and mild solution is standard (see \emph{e.g.} \citeauthor*{briand2008bsdes} \cite{briand2008bsdes}), and thus our stochastic representation of the value function in \Cref{prop:BSDE} immediately implies that \eqref{DPE} is satisfied in the mild sense.

\section{Case of singular kernels}\label{sec:singular-kernel}

In this section, we explain why the singular-kernel case is not excluded merely by probabilistic well-posedness. Consider, for instance, a Volterra process of the form
\[
X_t = \int_0^t K(t-r)\mathrm{d}W_r,\qquad t\in[0,T],
\]
with $K:(0,T]\longrightarrow \dbR$ singular at the origin. Such kernels naturally arise in rough-volatility models and in related control problems; see, for instance, \citeauthor*{fouque2019optimal} \cite{fouque2019optimal}. They also already appear in linear--quadratic control problems; see \citeauthor*{abi2021linear} \cite{abi2021linear}.

\medskip
From the viewpoint of existence and uniqueness of the stochastic Volterra dynamics itself, the literature is by now much richer than in the regular case considered here. Weak solution theories are available for convolution kernels \cite{abijaber2021weak}, for more general non-convolution kernels \cite{promel2023existence,abijaber2025weak}, and strong or pathwise-uniqueness results are known in some singular Hölder settings \cite{promel2023stochastic,promel2025pathwise}; see also \cite{hamaguchi2025weak} for completely monotone kernels.

\medskip
The main obstruction for the present paper is of a different nature. Our dynamic-programming approach relies on a rectangular lift $s\longmapsto X_t^s$ with values in a Sobolev/Hilbert space on $[0,T]$. For singular kernels, the natural parameter domain is typically triangular, and even after a rectangular extension the lifted paths usually belong at best to a Banach path space rather than to a Sobolev space. {Let us illustrate this difficulty with the example of a fractional Brownian motion
\[
 B_t^h\coloneqq \int_0^t (t-s)^{-h} \mathrm{d}W_s,
\]
with $h \in (0,1)$. First, to fit into our setting, we must extent the kernel to the rectangular time space $[0,T]^2$, which we can achieve by setting $\Si_t(s) \coloneqq  (t-s)^{-H}\1_{s < t}$. Next, we need to ensure that, for all $t \in [0,T]$, both $\Si_t$ and its weak derivative $\pa_s \Si_t(s) = -H(t-s)^{-H-1}\1_{s < t}$ are in some weighted $\dbL^2$ space. 

\medskip
However, we face the major difficulty here: a weighted measure is needed to ensure integrability of $\Si_t(\cdot)$ and $\pa_s \Si_t(\cdot)$ at the point where the kernel explodes, that is, when $s \uparrow t$. Since this point is $t$, the weighted measure, and therefore the lifting space for $\Si_t$ would depend on $t$; yet, our lifting approach requires the coefficients to be lifted in the same Hilbert space $H$ for all $t \in [0,T]$.
} 

\medskip
One may still derive formal path-dependent equations in some convenient Banach space, in the spirit of \citeauthor*{viens2019martingale} \cite{viens2019martingale}, but the Hilbert-space viscosity theory used in \Cref{sec:DPE} is no longer available in a form adapted to such lifts. Extending the dynamic-programming approach to singular kernels therefore requires not only probabilistic well-posedness of the underlying stochastic Volterra equation, but also a different lift and a matching viscosity/comparison theory on the resulting path space. {Another possible approach consists in considering regularised kernels; for example, going back to the example of the fractional Brownian motion above, one could defined
\[
 B_t^{h, \e} \coloneqq \int_0^t (t-s+\e)^{-h} \mathrm{d}W_s,
 \]
and naturally define the corresponding lift by setting $\Si_t(s) := (t-s+\e)^{-h}\1_{s < t} + \e^{-h}\1_{s \ge t}$. This approach is, for example, followed in \citeauthor*{di2019approximation} \cite{di2019approximation}, and can be applied to approximate control problems involving singular kernels by problems with regular kernels. Consider for example:
\[
 V_0 \coloneqq \sup_{\alpha} \dbE\big[g(X_T^\alpha)\big], 
 \]
with 
\[
X_t^\alpha = x + \int_0^t \phi(t-s)\big(b(X_s^\a, \a_s)\mathrm{d}s + \si(X_s^\a, \a_s)\mathrm{d}W_s\big),
\]
and $\phi$ such that $\phi(t) \longrightarrow \infty$ as $t \to 0$. We introduce the regular Volterra control problem
\[
V_0^\e \coloneqq \sup_{\alpha} \dbE\big[g(X_T^{\e,\alpha})\big], 
\]
with 
\[
X_t^{\e,\alpha} = x + \int_0^t \phi^\e(t-s)\big(b(X_s^{\e,\a,} \a_s)\mathrm{d}s + \si(X_s^{\e,\a}, \a_s)\mathrm{d}W_s\big),
\]
with $\phi^\e(\cdot) \coloneqq \phi(\cdot + \e)$. Observe that $\phi^\e$ does not explode at $0$, and therefore the problem $V_0^\e$ is directly covered by our theory. Then, assuming sufficient regularity on $g$, $b$ and $\si$, we may follow \citeauthor*{alfonsi2024approximation} \cite{alfonsi2024approximation} and show that
\[
 | V_0 - V_0^\e | \le C \int_0^T |\phi(s) - \phi^\e(s) |^2\mathrm{d}s,
 \]
which induces (under appropriate regularity and integrability conditions on $\phi$) that $V_0^\e \longrightarrow V_0$ as $\e \longrightarrow 0$.   
}

\appendix

\section{A class of stochastic target problems in Hilbert spaces}
\label{sec:appen}

We briefly extend the stochastic-target arguments of \citeauthor*{soner2002stochastic} \cite{soner2002stochastic} to an infinite-dimensional Hilbert setting. In this appendix, $H$ denotes an arbitrary Hilbert space, $W$ a standard one-dimensional Brownian motion, and $\dbA$ an open convex subset of a separable Banach space. Let $\cA$ be the set of càdlàg progressively measurable $\dbA$-valued controls.

\medskip
For $(t,\bx,y)\in[0,T]\times H\times\dbR$ and $\a\in\cA$, consider the SDEs
\begin{align}\label{SDEtarget}
\mathrm{d}\bX_r^{t,\bx,\a} &= B\big(r,\bX_r^{t,\bx,\a},\a_r\big)\mathrm{d}r + \Si\big(r,\bX_r^{t,\bx,\a},\a_r\big)\mathrm{d}W_r, \nonumber\\
\mathrm{d}Y_r^{t,\bx,y,\a} &= b\big(r,\bX_r^{t,\bx,\a},Y_r^{t,\bx,y,\a},\a_r\big)\mathrm{d}r + \si\big(r,\bX_r^{t,\bx,\a},Y_r^{t,\bx,y,\a},\a_r\big)\mathrm{d}W_r,
\end{align}
with initial condition $(\bX_t^{t,\bx,\a},Y_t^{t,\bx,y,\a})=(\bx,y)$. We assume throughout that $(B,\Si)$ and $(b,\si)$ satisfy the standard Lipschitz and linear-growth assumptions ensuring well-posedness of \eqref{SDEtarget}.

Given $g:H\longrightarrow\dbR$, define
\[
 \cA(t,\bx,y) \coloneqq \big\{ \a \in \cA : Y_T^{t,\bx,y,\a} \ge g(\bX_T^{t,\bx,\a}) \big\}.
\]
By monotonicity of $Y$ with respect to its initial condition, and because $\bX$ does not depend on $y$, we have
\[
 \cA(t,\bx,y) \neq \emptyset \Longrightarrow \cA(t,\bx,y^\prime)\neq\emptyset,\; \forall y^\prime \ge y.
\]
This leads to the value function
\begin{equation}\label{target-value-function}
w(t,\bx) \coloneqq \inf\big\{ y \in \dbR : \cA(t,\bx,y) \neq \emptyset\big\}.
\end{equation}

As in \cite{soner2002dynamic}, the dynamic programming principle reads as follows.

\begin{proposition}\label{DPPgeo}
For all $(t,\bx)\in[0,T]\times H$ and all $[t,T]$-valued stopping times $\th$, we have
\begin{equation}\label{target-DPP}
w(t,\bx) = \inf\big\{ y \in \dbR : \exists \a \in \cA,\; Y_\th^{t,\bx,y,\a} \ge w(\th,\bX_\th^{t,\bx,\a}),\; \dbP\text{\rm--a.s.}\big\}.
\end{equation}
\end{proposition}

\begin{proof}
The proof of \cite[Theorem 3.1]{soner2002dynamic} extends verbatim to the present setting: the state process $(\bX,Y)$ is strong Markov, the control set is separable and stable under concatenation, and the measurable-selection step still follows from the Jankov--von Neumann theorem.
\end{proof}

For $(t,\bx,y,\bp)\in[0,T]\times H\times\dbR\times H$, define
\[
 \cN(t,\bx,y,\bp) \coloneqq \big\{ a\in\dbA : \si(t,\bx,y,a)-\langle \Si(t,\bx,a),\bp\rangle_H = 0 \big\}.
\]
We consider the PDE
\begin{equation}\label{target-DPE}
- \pa_t u(t,\bx) + \sup_{a \in \cN(t,\bx,u(t,\bx),D_\smalltext{\bx}u(t,\bx))} \Big\{ b\big(t,\bx,u(t,\bx),a\big) - \cL^a u(t,\bx) \Big\} = 0,\qquad u(T,\bx)=g(\bx),
\end{equation}
where
\[
\cL^a u(t,\bx) \coloneqq \langle B(t,\bx,a),D_\bx u(t,\bx)\rangle_H + \frac12 \langle \Si(t,\bx,a),D_{\bx\bx}^2u(t,\bx)\Si(t,\bx,a)\rangle_H.
\]

\begin{definition}[Viscosity solutions]
Let $u:[0,T]\times H\longrightarrow\dbR$ be continuous.

\medskip
$(i)$ $u$ is a viscosity super-solution of \eqref{target-DPE} if $u(T,\cdot)\ge g$ and, for every $\f\in C^{1,2}([0,T]\times H)$ such that $u-\f$ has a local minimum at $(t,\bx)$,
\[
- \pa_t \f(t,\bx) + \sup_{a \in \cN(t,\bx,u(t,\bx),D_\smalltext{\bx}\f(t,\bx))} \Big\{ b\big(t,\bx,u(t,\bx),a\big) - \cL^a \f(t,\bx) \Big\} \ge 0.
\]

$(ii)$ $u$ is a viscosity sub-solution of \eqref{target-DPE} if $u(T,\cdot)\le g$ and, for every $\f\in C^{1,2}([0,T]\times H)$ such that $u-\f$ has a local maximum at $(t,\bx)$,
\[
- \pa_t \f(t,\bx) + \sup_{a \in \cN(t,\bx,u(t,\bx),D_\smalltext{\bx}\f(t,\bx))} \Big\{ b\big(t,\bx,u(t,\bx),a\big) - \cL^a \f(t,\bx) \Big\} \le 0.
\]

$(iii)$ $u$ is a viscosity solution of \eqref{target-DPE} if it is both a viscosity super-solution and a viscosity sub-solution.
\end{definition}

\begin{theorem}\label{target-EDP}
Assume that $w$ is continuous, and that $\cN$ is continuous in the sense that if $a_0 \in \cN(t_0, \bx_0, y_0, \bp_0)$, then there exists a mapping $\hat a : [0,T] \times H \times \dbR \times H \longrightarrow \dbA$ such that
\[\begin{cases}
\displaystyle\hat a(t_0, \bx_0, y_0, \bp_0) = a_0, \\
\displaystyle\hat a(t, \bx, y, \bp) \in \cN(t, \bx, y, \bp) \ \forall (t,\bx, y, \bp) \in [0,T] \times H \times \dbR \times H.
\end{cases}
\]
Then $w$ is a viscosity solution of \eqref{target-DPE}. 
\end{theorem}
\begin{proof} 
The argument is very similar to \cite{soner2002stochastic}. However, as we are in an infinite dimensional setting and that we do not require $\dbA$ to be compact, we detail the proof.

\medskip
$(i)$ We first show the super-solution property. Let $\f$ be a test function; we may assume without loss of generality that $\f(t, \bx) = w(t, \bx)$ and that the minimum in the tangency property is global. Let $\a \in \cA(t, \bx, w(t, \bx))$ and introduce, for $\d > 0$, 
\begin{equation}\label{target-localisation}
 \th_\d \coloneqq \inf\big\{s \ge t : (s, \bX_s^{t, \bx, \a}) \notin [t,t+\d) \times B_\d(\bx)\big \}, 
 \end{equation}
where $B_\d(\bx)$ is the ball of radius $\d$ and centre $\bx$ in $H$.  As a consequence of the DPP \eqref{target-DPP}, we have
\[
Y_{\th_\d}^{t, \bx, w(t,\bx), \a} \ge w\big(\th_\d, \bX_{\th_\d}^{t, \bx, w(t,\bx), \a}\big) \ge \f\big(\th_\d, \bX_{\th_\d}^{t, \bx, w(t,\bx), \a}\big), \; \dbP\mbox{\rm--a.s.}. 
\]
Therefore, applying Itô's formula between $t$ and $\th_\d$, we obtain
\begin{align*}
&\int_t^{\th_\d} \Big(b\big(s, \bX_s^{t,\bx,\a}, Y_s^{t,\bx,w(t, \bx),\a}, \a_s\big) - \pa_t \f\big(s, \bX_s^{t, \bx, \a}\big) - \cL^{\a_s} \f\big(s, \bX_s^{t, \bx, \a}\big)\Big)\mathrm{d}s \\
&\quad+ \int_t^{\th_\d} \Big( \si\big(s, \bX_s^{t,\bx,\a}, Y_s^{t,\bx,w(t, \bx),\a}, \a_s\big) -  \langle \Si\big(s, \bX_s^{t, \bx, \a}, \a_s\big), D_\bx \f\big(s,\bX_s^{t, \bx, \a}\big) \rangle  \Big)\mathrm{d}W_s \ge 0,
\end{align*}
$\dbP$--a.s. For $n \in \dbN$, we now introduce the measure $\dbP^n$ defined by
\[
 \frac{\mathrm{d}\dbP^n}{\mathrm{d}\dbP} \coloneqq \cE\bigg(-n \int_t^{T \wedge \th_\smalltext{\d}} \big( \si(s, \bX_s^{t,\bx,\a}, Y_s^{t,\bx,w(t, \bx),\a}, \a_s) - \Si(s, \bX_s^{t, \bx, \a}, \a_s) \big)\mathrm{d}W_s\bigg), \]
so that taking the expectation under $\dbP^n$ in the above inequality provides
\begin{align*}
&\dbE^{\dbP^n}\bigg[\int_t^{\th_\smalltext{\d}} \Big(b\big(s, \bX_s^{t,\bx,\a}, Y_s^{t,\bx,w(t, \bx),\a}, \a_s\big) -  \pa_t \f\big(s, \bX_s^{t, \bx, \a}\big) - \cL^{\a_s} \f\big(s, \bX_s^{t, \bx, \a}\big)\Big)\mathrm{d}s \\
&\quad -n \int_t^{\th_\smalltext{\d}} \Big( \si\big(s, \bX_s^{t,\bx,\a}, Y_s^{t,\bx,w(t, \bx),\a}, \a_s\big) - \langle \Si\big(s, \bX_s^{t, \bx, \a}, \a_s\big), D_\bx \f\big(s,\bX_s^{t, \bx, \a}\big) \rangle  \Big)^2 \mathrm{d}s \bigg] \ge 0.
\end{align*}
Then, {as $\alpha$ is required to be right-continuous, we may use the integral mean value theorem} and the fact that this inequality must be true for all $n \in \dbN$ to obtain
\[
 b\big(s, \bx, w(t,\bx), \a_t\big) - \pa_t \f(s, \bx) - \cL^{\a_t} \f(s, \bx) \ge 0, 
 \]
with $\a_t$ such that 
\[
 \big( \si(t, \bx, w(t, \bx), \a_t) - \langle \Si(t, \bx, \a_t), D_\bx \f(t,\bx) \rangle  \big)^2 = 0,
 \]
which proves the super-solution property. 

\medskip
$(ii)$ We now prove the sub-solution property. Given a test function $\f$, we may assume without loss of generality that $\f(t, \bx) = w(t, \bx)$ and that 
\begin{equation}\label{ineq-target1}
\inf_{(s, \tilde \bx) \in \pa_\smalltext{p} B_\d(t, \bx)}\f(s, \tilde \bx) - w(s, \tilde \bx) \ge \d > 0,
\end{equation}
where $\pa_pB_\d(t, \bx) \coloneqq \{t + \d\} \times {\rm cl}(B_\d(\bx)) \cup [t, t+\d] \times \pa B_\d(\bx)$ is the parabolic border of $B_\d(t, \bx) \coloneqq [t,t+\d) \times B_\d(\bx)$. This can be achieved for example by adding a term in $\lvert \tilde \bx - \bx \rvert_H^4$ to the test function. 

\medskip
We shall prove the sub-solution property by contradiction. Assume that 
\begin{equation}\label{contradiction}
- \pa_t u(t, \bx) + \sup_{a \in \cN(t, \bx, u(t,\bx), D_\smalltext{\bx} u(t,\bx))} \big\{ b(t, \bx, u(t,\bx),a) - \cL^a u(t,\bx) \big\} > 0.
\end{equation}
By continuity of $\cN$, $\d$ can be chosen so that 
\[
 - \pa_t u(t, \bx) + \big\{ b\big(t, \bx, u(t,\bx),\hat a(s, \tilde \bx, \f(s, \tilde \bx), D_\bx f(s, \tilde \bx) )\big) - \cL^{a(s, \tilde \bx, \f(s, \tilde \bx), D_\smalltext{\bx} f(s, \tilde \bx) )} u(t,\bx) \big\} \ge \d ,
 \]
 for all $(s, \tilde \bx) \in B_\d(t, \bx)$, with the mapping $\hat a$ as in the assumptions of the theorem. Fix now $\eta > 0$, and let $(\bX^\eta, Y^\eta)$ be the solution of the SDEs \eqref{SDEtarget} such that 
 \[
  \bX_t^\eta = \bx, \ Y_t^\eta = w(t,\bx) - \eta, 
  \]
 and controlled by $\hat \a_s \coloneqq  \hat a(s, \bX_s^\eta, Y_s^\eta, D_\bx \f(s, X_s^\eta))$. Let also $\th_\d$ be as in \eqref{target-localisation}. We have
 \begin{align*}
 Y_{\th_\d}^\eta - w(\th_\d,  \bX_{\th_\d}^\eta) &=  Y_{\th_\d}^\eta - \f(\th_\d,  \bX_{\th_\d}^\eta)  + \f(\th_\d,  \bX_{\th_\d}^\eta) - w(\th_\d,  \bX_{\th_\d}^\eta) \ge  Y_{\th_\d}^\eta - \f(\th_\d,  \bX_{\th_\d}^\eta) + \d,
 \end{align*}
 by \eqref{ineq-target1}. Introduce now the process $\hat Y_s^\eta \coloneqq\f(s, \bX_s^\eta)-\eta$. As in \cite{soner2002stochastic}, we observe that $\hat Y^\eta$ satisfies the same SDE as $Y^\eta$ with a lower drift term, due to our hypothesis \eqref{contradiction}. Therefore, by stochastic comparison, we have $\hat Y^\eta \le Y^\eta$. Coming back to the previous inequalities, we have
 \begin{align*}
  Y_{\th_\d}^\eta - w(\th_\d,  \bX_{\th_\d}^\eta) &\ge Y_{\th_\d}^\eta - \hat Y_{\th_d}^\eta +  \hat Y_{\th_d}^\eta - \f(\th_\d,  \bX_{\th_\d}^\eta) + \d \ge - \eta + \d.
 \end{align*}
Thus, taking $\eta < \d$, we obtain a contradiction of the DPP \eqref{target-DPP}. Thus \eqref{contradiction} is false and the viscosity sub-solution property is satisfied. 
\end{proof}

\bibliography{bibliographyDylan}

\end{document}